\documentclass{amsart}

\usepackage[utf8]{inputenc}
\usepackage[english]{babel}
\usepackage{microtype}
\usepackage{amssymb,amsfonts,amsmath,latexsym,amsthm}
\usepackage{xcolor}
\definecolor{darkblue}{rgb}{0,0,0.6}
\usepackage{graphicx}
\usepackage[all,cmtip]{xy}
\usepackage[ocgcolorlinks,colorlinks=true, citecolor=darkblue, filecolor=darkblue, linkcolor=darkblue, urlcolor=darkblue]{hyperref}
\usepackage[nameinlink, capitalise]{cleveref}
\usepackage[normalem]{ulem}

\hypersetup{bookmarksopen=true}

\usepackage{array}

\usepackage{tikz}
\usetikzlibrary{arrows.meta}
\usetikzlibrary{cd}

\usepackage{mathrsfs}
\usepackage{bm} 

\usepackage{slashed}

\newtheorem{theorem}{Theorem}[section] 
\newtheorem{prop}[theorem]{Proposition}
\newtheorem{lem}[theorem]{Lemma}
\newtheorem{kor}[theorem]{Corollary}
\newtheorem{con-alt}[theorem]{Conjecture}
\newtheorem{constr-alt}[theorem]{Construction}
\theoremstyle{definition}
\newtheorem{ddd-alt}[theorem]{Definition}
\newtheorem{ass-alt}[theorem]{Assumption}
\newtheorem{prob-alt}[theorem]{Problem}

\newenvironment{ddd}    
{%
	\pushQED{\qed}\begin{ddd-alt}}
	{\popQED\end{ddd-alt}}

\newenvironment{ass}    
{%
	\pushQED{\qed}\begin{ass-alt}}
	{\popQED\end{ass-alt}}

{%
	\pushQED{\qed}\begin{con-alt}}
	{\popQED\end{con-alt}}

{%
	\pushQED{\qed}\begin{constr-alt}}
	{\popQED\end{constr-alt}}

{%
	\pushQED{\qed}\begin{prob-alt}}
	{\popQED\end{prob-alt}}

\theoremstyle{remark}

\theoremstyle{definition}
\newtheorem{ex-alt}[theorem]{Example}
\newtheorem{rem-alt}[theorem]{Remark}

\newenvironment{ex}    
{%
	\pushQED{\qed}\begin{ex-alt}}
	{\popQED\end{ex-alt}}

\newenvironment{rem}    
{%
	\pushQED{\qed}\begin{rem-alt}}
	{\popQED\end{rem-alt}}

\crefname{theorem}{Theorem}{Theorems}
\crefname{lem}{Lemma}{Lemmas}
\crefname{prop}{Proposition}{Propositions}
\crefname{section}{Section}{Sections}
\crefname{ex-alt}{Example}{Examples}
\crefname{ddd-alt}{Definition}{Definitions}
\crefname{kor}{Corollary}{Corollaries}

\numberwithin{equation}{section}
\newcommand{\ab}{\mathrm{ab}}
\newcommand{\bA}{\mathbf{A}}

\newcommand{\bC}{\mathbf{C}}

\newcommand{\C}{\mathbb{C}}

\newcommand{\K}{\mathbb{K}}

\newcommand{\nat}{\mathbb{N}}

\newcommand{\R}{\mathbb{R}}

\newcommand{\Z}{\mathbb{Z}}

\newcommand{\cA}{\mathcal{A}}
\newcommand{\cB}{\mathcal{B}}
\newcommand{\cC}{\mathcal{C}}

\newcommand{\cO}{\mathcal{O}}

\newcommand{\cV}{\mathcal{V}}

\newcommand{\cX}{\mathcal{X}}

\DeclareMathOperator{\Cofib}{Cofib}

\DeclareMathOperator{\dom}{dom}

\DeclareMathOperator{\Hom}{Hom}
\DeclareMathOperator{\Homol}{Hg}

\DeclareMathOperator{\im}{im}

\DeclareMathOperator{\Ind}{Ind}

\DeclareMathOperator{\Li}{\Li}

\DeclareMathOperator{\Res}{Res}

\DeclareMathOperator{\supp}{supp}

\DeclareMathOperator{\tr}{tr}

\renewcommand{\emptyset}{\varnothing}

\newcommand{\BC}{\mathbf{BornCoarse}}

\newcommand{\Coarse}{\mathbf{Coarse}}

\newcommand{\Sp}{\mathbf{Sp}}

\newcommand{\UBC}{\mathbf{UBC}}

\newcommand{\loc}{{\mathrm{loc}}}

\newcommand{\ch}{{\mathbf{ch}}}
\newcommand{\Dirac}{\slashed{D}}
\newcommand{\ind}{{\mathtt{index}}}

\newcommand{\Orb}{\mathbf{Orb}}

\newcommand{\Hilb}{\mathbf{Hilb}}
\newcommand{\inter}{\mathrm{int}}

\renewcommand{\hat}{\widehat}
\renewcommand{\tilde}{\widetilde}

\begin{document}
  \title[Breaking symmetries]{Breaking symmetries for equivariant coarse homology theories}

\author[U.~Bunke]{Ulrich Bunke}
\address{Fakult{\"a}t f{\"u}r Mathematik,
	Universit{\"a}t Regensburg,
	93040 Regensburg,
	Germany}
\email{ulrich.bunke@mathematik.uni-regensburg.de}

\author[M.~Ludewig]{Matthias Ludewig}
\address{Fakult{\"a}t f{\"u}r Mathematik,
	Universit{\"a}t Regensburg,
	93040 Regensburg,
	Germany}
\email{matthias.ludewig@mathematik.uni-regensburg.de}

 \date{\today}

\begin{abstract}
We describe a symmetry breaking construction in coarse geometry   which allows to obtain information about equivariant coarse homology classes 
by restriction to smaller groups and spaces.
  In the case of equivariant coarse $K$-homology theory we give an analytic interpretation of this construction. As a consequence we obtain applications to the spectral theory of invariant differential operators. 
 \end{abstract}
\maketitle
\tableofcontents
\section{Introduction}

  We consider a  differential operator  $D$ of Dirac or Laplace type  acting on sections of  a Hermitian vector bundle on  a complete Riemannian manifold $M$. We assume   the operator  is invariant under the action of a  
 {discrete}
  group of symmetries $G$.   
According to \cref{weroigjoergergegergwegergg} a   spectral interval  for $D$ is an  interval   in $\R$ whose endpoints are located in spectral gaps of $D$.
The spectral projection  $E_{D}(I)$   associated to  a  bounded spectral interval {$I$}   {is an element of the} $G$-invariant Roe algebra $C^{G}(M)$  associated to the underlying coarse space of $M$, see \cref{qoijfoqefeqewfwefeqwfe}.  

We now consider a  {subset} $Z$ with smooth boundary $\partial^{\infty} Z$ which are both preserved by the action of a subgroup $K$ of $G$.
If we define    a selfadjoint extension of the restriction $D_{Z}$ to $Z$ using $K$-invariant elliptic local boundary conditions at $\partial^{\infty} Z$,  then for a bounded spectral interval {$I$}, 
 $E_{D_{Z}}(I)$  {is an element of} the  $K$-invariant Roe algebra $C^{K}(Z)$ associated to  
the underlying coarse space of $Z$, see \cref{qoijfoqefeqewfwefeqwfe}.

The   operator $D_{Z}$ is   ``locally $G$-invariant'' as a differential operator, but as a functional analytic object its invariance is distorted  by the boundary condition near $\partial^{\infty} Z$.
{However,} it turns out that the influence of this distortion on the spectral projections decays with the distance from $\partial^{\infty} Z$.  
In particular, they  become  {``}more and more $G$-invariant{''} if we move far {away} from the boundary. 
In order to capture this behavior in a precise manner,  under a wideness assumption (see \cref{9wtgiohwgregwerg}) on $Z$ in \cref{qwefihfuiewqfeefewfewfeqwfqfeceq} we  introduce an extension of Roe algebras
\begin{equation}\label{ewqfojqwfqewfqewfe}
0\to    C^{K}(\partial Z)\to  C^{G,K}(M,Z)\stackrel{\sigma}{\to} C^{G}(M)\to 0\ ,
\end{equation}
where $C^{K}(\partial Z)$ is the subalgebra of $C^{K}(Z)$ generated by $K$-invariant operators supported near $\partial^{\infty} Z$, and
$ C^{G,K}(M,Z)$ is a {certain} subalgebra of $C^{K}(Z)$ {consisting of asymptotically $G$-invariant operators}.
 {In the special case that $K$ is trivial, this extension has previously been considered in \cite{LudewigThiangCobordism}.}
The precise statement is now that  $E_{D_{Z}}(I)$ (apriori  belonging to $C^{K}(Z)$)  belongs to the subalgebra $  C^{G,K}(M,Z)$, and that 
\begin{equation}\label{qewr09i0qwerqewrqr}
\sigma (E_{D_{Z}}(I))=E_{D}(I)\ , 
\end{equation}
see \cref{eirogjwoegwergerwf}.

Applying $K$-theory to the exact sequence \eqref{ewqfojqwfqewfqewfe} of $C^{*}$-algebras we get a long exact sequence
\begin{equation}\label{gergergregregerggwr}
\cdots \to K_{*}( C^{G,K}(M,Z))\stackrel{K(\sigma)}{\to} K_{*} ( C^{G}(M))\stackrel{\delta}{\to}
K_{*-1}(C^{K}(\partial Z))\to \cdots
\end{equation} 
of $K$-theory groups.
 A projection $P$  in a $C^{*}$-algebra represents a $K$-theory class  $[P]$ in the degree zero $K$-theory group. {As first observed in \cite[Thm.\ 3.4]{LudewigThiangCobordism} (see also \cite[\S 2.2.3]{ThiangEdgeFollowing})}, the exactness of 
 \eqref{gergergregregerggwr} has the following simple consequence (see \cref{09qerug09egwgwerge}).

 \begin{kor}\label{weoigjoegergwegwregrewgw}
 If $I$ is a spectral interval of $D$ and $\delta[E_{D}(I)]\not=0$, then $I$ is not a spectral interval of $D_{Z}$.
 \end{kor}
 
Note that the conclusion of the corollary means that at least one of the endpoints of $I$
belongs to the spectrum of $D_{Z}$. The corollary can therefore be employed to show that certain numbers must belong to the spectrum of $D_{Z}$. {Such statements are important in the theory of topological insulators, where one is interested in the existence of certain ``boundary-localised'' states that arise precisely from a spectral gap in the bulk Hamiltonian that is filled when introducing a boundary. For a coarse geometric discussion of these phenomena, see {\cite{EwertMeyer}, \cite{ThiangEdgeFollowing}, \cite{LudewigThiangCobordism} and \cite{LudewigThiangGapless}.}}

We will demonstrate this sort of application in the case of magnetic Hamiltonians in \cref{qeroighjwoergwregrweg9}.

In order to apply   \cref{weoigjoegergwegwregrewgw} 
 we must be able to calculate the boundary operator $\delta$ {of the long exact sequence} in \eqref{gergergregregerggwr}.  
 The main idea of the present paper is  
 to consider the construction of the sequence  \eqref{gergergregregerggwr} as a special
 case of a general construction in equivariant coarse homotopy theory. 

\medskip

The main objects of coarse homotopy theory  as developed in \cite{equicoarse} are equivariant coarse homology theories
$$E^{G}:G\BC\to \bC\ ,$$ where $G\BC$ is the category of $G$-bornological coarse spaces and the target $\bC$ is some stable $\infty$-category {such as the category of spectra}.  We consider a $G$-bornological coarse space   $X$  with a $K$-invariant subset  $Z$. In \cref{wegoijwegerwregwg}
 we define the obstruction morphism
$$r: E^{G}(X)\stackrel{c^{*}}{\to} E^{K}(\Res^{G}_{K}(X))\stackrel{\delta^{MV}}{\to} \Sigma E^{K}(\partial Z)\ .$$
Here $\Res^{G}_{K}(X)$ is the space $X$ considered as a $K$-bornological coarse space via the restriction along $K\to G$.
The symbol $E^{K}$ denotes the $K$-equivariant coarse homology theory which is naturally  derived from $E^{G}$ using a transfer structure, see \eqref{NotationEK}. The map 
$\delta^{MV}$ is the Mayer-Vietoris boundary map associated to the decomposition of $X$ into $Z$ and its complement
$\{X\setminus Z\}$, the big family  consisting of the coarse thickenings of $X\setminus Z$.
 Finally, the big family $\partial Z:=Z\cap \{X\setminus Z\}$ is the coarse version of the boundary of $Z$ (we refer \cite{equicoarse} for the basic definitions from equivariant coarse homotopy theory).
Since $\bC$ is a stable $\infty$-category we can extend the obstruction morphism $r$ to a {fibre} sequence \begin{equation}\label{}
E^{G,K}(X,Z)\to E^{G}(X)\stackrel{r}{\to} E^{K}(\partial Z)
\end{equation}
defining the relative object $E^{G,K}(X,Z)$, see \cref{wfeioweedsfd}.

We can apply these constructions to the equivariant coarse $K$-homology $K\cX^{G}$ {introduced in \cite{coarsek}} in place of $E^{G}$ and the manifold $M$ from above in place of $X$.
Except for degenerate cases we have natural isomorphisms of $K$-theory groups
$$K\cX^{G}_{*}(M)\cong K_{*}(C^{G}(M))\ , \quad K\cX^{K}_{*}(\partial Z)\cong K_{*}(C^{K}(\partial Z)) \ , $$
{see} \cite[Thm. 6.1]{indexclass}.
The following  theorem is a reformulation of   \cref{eiogwegergwegergwerg} in terms of homotopy groups.
  \begin{theorem}
 There exists a canonical isomorphism of long exact sequences
 $$\xymatrix{\ar[r]&K_{*}(C^{G,K}(M,Z))\ar[r]\ar[d]^{\cong}&K_{*}(C^{G}(M))\ar[d]^{\cong}\ar[r]^{\delta}&K_{*-1}(C^{K}(\partial Z))\ar[d]^{\cong}\ar[r]&\\\ar[r]&K\cX^{G,K}_{*}(M,Z)\ar[r]&K\cX^{G}_{*}(M)\ar[r]^{r_{*}}&K\cX^{K}_{*-1}(\partial Z)\ar[r]&}$$
 \end{theorem}
 We interpret this theorem and in particular the commutativity of the  right square  as a kind of index theorem relating the purely analytic construction of $\delta $ in terms of Roe algebras with the purely homotopy theoretic construction of $r_{*}$.
 The theorem reduces the calculation of $\delta$ to the calculation of the obstruction morphism $r$.

\medskip

 {\em Acknowledgements: U.\ Bunke  and M.\ Ludewig  were supported by the SFB 1085 ``Higher Invariants''
funded by the Deutsche Forschungsgemeinschaft DFG. }

\section{A homotopy theoretic construction}\label{wrthwoiejgioergrewg}

We consider a group $G$.  
Coarse geometry was invented by J. Roe \cite{MR1147350}, \cite{roe_lectures_coarse_geometry} as a way to capture the large-scale geometry
of metric spaces with an isometric $G$-action. 
In the present paper we use the formalization
in terms of $G$-bornological coarse spaces introduced in \cite{equicoarse}, see  \cite{Bunke:2023aa} for an overview. A $G$-bornological coarse space is a $G$-set {$X$} with a $G$-coarse structure consisting
of a collection of subsets $U$ of $X \times X$, called
coarse entourages, and a compatible $G$-invariant bornology determining the collection of bounded subsets of $X$. Morphisms between $G$-bornological coarse spaces are $G$-invariant maps which are controlled in the sense that they send coarse  entourages to coarse entourages and proper in the sense that preimages of bounded sets are bounded. 
The category of $G$-bornological coarse spaces and invariant controlled and proper maps is called the category of $G$-bornological coarse spaces, {denoted} $G\BC$.

Typical examples of $G$-bornological coarse spaces are the group $G$ itself with the left-action {on itself}, the minimal bornology and the minimal or the canonical  $G$-coarse structures, denoted by $G_{min,min}$ and $G_{can,min}$, respectively. Here the canonical coarse structure is the smallest $G$-coarse structures containing all
enourages $\{(g,g')\}$ for pairs $g,g'$ in $G$.

A basic tool to  construct coarse invariants of $G$-bornological coarse spaces are equivariant coarse homology theories. 
In the present paper we use the axiomatization introduced in \cite{equicoarse},  see also   \cite{Bunke:2023aa}. A $G$-equivariant coarse homology theory is a functor $$E^{G}:G\BC\to \bC $$
to some cocomplete stable $\infty$-category ${\bC}$ which is coarsely invariant, excisive, annihilates flasques,
and is $u$-continuous. For a detailed description of these properties we refer to \cite{equicoarse}. 
For the present paper the most important example is the equivariant coarse $K$-homology $K\cX^{G}$ constructed in \cite{coarsek}.

Given a
 $G$-equivariant coarse homology theory $$E^{G}:G\BC\to \bC $$  and a $G$-bornological coarse space $X$, in the present section we describe {a}  construction which may produce information about $E^{G}(X)$ by breaking the symmetry. 
Our main result   is the construction of the relative object $E^{G,K}(X,Z)$ in \cref{wfeioweedsfd} and the fibre sequence \eqref{fsfdkijvoirjovwevewvervwv}.

Let $K$ be a subgroup of $G$ and consider the {induction  functor}
$$\Ind_{K}^{G}:K\BC\to G\BC$$ 
sending a $K$-bornological coarse space $X$ to the $G$-set $G \times_K X$ with the appropriate coarse and bornological structures
\cite[Sec.\ 6.5]{equicoarse}.
We  define an  associated $K$-equivariant coarse homology $E^{K}$  by 
\begin{equation}\label{NotationEK}
E^{K}:=E^{G}\circ \Ind_{K}^{G}:K\BC\to \bC
\end{equation}
  \cite[Lemma 4.19]{desc}.
In the case of the trivial group $K=\{1\}$ we will 
{write simply $E$ instead of $E^K$}.
In this case the induction functor is given in terms of the symmetric monoidal structure of $G\BC$ by 
$ \Ind^{G}(X)\cong G_{min,min}\otimes \Res_{G}(X)$, where $\Res_{G}$ equips the bornological coarse space $X$ with the trivial $G$-action.
We therefore 
 have the equivalence  \begin{equation}\label{ewqfqefopkpoqwefqwefewfwqf}
E(X)  \simeq  E^{G}(G_{min,min}\otimes \Res_{G}(X))\ .
\end{equation}
 \begin{rem}\label{wfrrefwerfergwrgwergwefref}
  Often we are in the situation that we already have a family of  equivariant coarse homology theories  $(E^{H})_{H}$  (with a fixed target)  for all {groups} $H$, which at first sight clashes with the notation \eqref{NotationEK}.
However, in all relevant examples the homology theory $E^{G}\circ \Ind_{K}^{G}$ derived from  the member $E^{G}$ of the family   using \eqref{NotationEK}
is equivalent to the member of the family denoted by the same symbol $E^{K}$. In this situation we will say that the family has induction equivalences.

 Examples of families with  induction equivalences are  the family of equivariant coarse ordinary homology theories   $(H\cX^{{H}})_{{H}}$ \cite[Prop. 7.12]{equicoarse} and the family of coarse algebraic $K$-theories  $(K\bA\cX^{{H}})_{{H}}$ associated to an additive category $\bA$ (with trivial action) \cite[Prop. 8.27]{equicoarse}.  {For the present paper, the relevant example is the coarse topological $K$-homology associated to a $C^{*}$-category with a 
$G$-action  obtained in \cite[Thm. 6.3]{coarsek} (for the $K$-theory   of $C^{*}$-categories in place of homological functor $\Homol$), see   \cref{thrertherhertheth}. {That this example has induction equivalences has been shown in}  \cite[Cor. 10.5.{1}]{coarsek}.}
 \end{rem}

We {also} have a {restriction functor}
$$\Res^{G}_{K}:G\BC\to K\BC$$
which sends a $G$-bornological space $X$ to the $K$-bornological space obtained by restricting the action from $G$ to $K$.
  If $X$ is a $G$-bornological coarse space, then we have a natural map of coarse spaces 
   \begin{equation}\label{wefqwefqewfqwefewfqwef}
 c_{X}:G\times_{K} X\to X\ , \quad  [g,x]\mapsto gx  \ .
\end{equation}
In general this {morphism in $G\Coarse$} is not proper (and hence not a morphism in $G\BC$),  but it is  a  {bounded covering} \cite[Def.\ 2.14]{coarsetrans}.  

We now assume that $E^{G}$ {admits transfers} \cite[Def.\ 1.2]{coarsetrans}.
{This means that $E^{G}$ has a contravariant functoriality with respect to bounded coverings in addition to the covariant functoriality
for morphisms in $G\BC$. The compatibility of these operations is encoded in the category $G\BC_{\tr}$ of $G$-bornological coarse spaces with transfers introduced in  \cite[Sec.2.2]{coarsetrans}.}
  The family of bounded coverings $c=(c_{X})$ can be considered as natural transformation of functors 
\begin{equation}\label{asdfasfdfasfadqew}
c:  \Ind_{K}^{G}\circ \Res^{G}_{K} \to    \iota : G\BC\to G\BC_{\tr}\ ,
\end{equation}  where $\iota:G\BC\to G\BC_{\tr}$ is the canonical embedding  \cite[Sec.2.33]{coarsetrans}.
{Applying $E^G$, we obtain a} natural transformation \begin{equation}\label{adsfawefdfa}
 {c^{G}_{K}}:E^{G}\to E^{G}\circ \Ind_{K}^{G}\circ \Res^{G}_{K}\stackrel{\eqref{NotationEK}}{\simeq} E^{K}\circ \Res^{G}_{K} \ .
\end{equation}

For a $K$-invariant subset $Y$ of $X$ we let $\{Y\}:=(U[Y])_{U\in \cC^{K}_{X}}$ denote the $K$-invariant  {big family} {\cite[Def.\ 3.5]{equicoarse}} consisting of all of $U$-thickenings of $Y$ for   $K$-invariant coarse  entourages $U$ in $\cC^K_X$.
Let {now}  $Z$ be a $K$-invariant subset of $X$ with the induced $K$-bornological coarse structure. Observe that $X \setminus Z$ is also $K$-invariant. We get  a $K$-invariant {complementary pair} $(Z,\{X\setminus Z\})$ on $\Res^{G}_{K}(X)$ {\cite[Def.\ 3.7]{equicoarse}}.  
We let $i:E^{K}(Z)\to E^{K}(\Res^{G}_{K}(X))$ and
$j:E^{K}(\{X\setminus Z\})\to E^{K}(\Res^{G}_{K}(X))$
denote the maps induced by the canonical inclusions.

 \begin{ddd}\label{wfeioweedsfd}
We define the object $E^{G,K}(X,Z)$ in $\bC$ by the pull-back 
\begin{equation*}
\begin{tikzcd}[column sep=1.5cm]
E^{G,K}(X,Z)\ar[r, "{e_{Z}\oplus e_{X\setminus Z}}", dashed] \ar[d, "s"', dashed]  & E^{K}(Z)\oplus E^{K}(\{X\setminus Z\})\ar[d, "i+j"]\\ E^{G}(X)\ar[r, "{c^{G}_{K}}"] &E^{K}(\Res^{G}_{K}(X))\ .
\end{tikzcd} \qedhere
\end{equation*}
 \end{ddd}

We consider the big family \begin{equation}\label{feorihioweregw}
\partial Z:=Z\cap \{X\setminus Z\}
\end{equation}
defined by intersecting the members of  the $K$-invariant big family $\{X\setminus Z\}$ with $Z$. It plays the {coarse geometric role of the boundary} of $Z$, hence the notation.

\begin{lem}
We have a canonical  fibre sequence
 \begin{equation}\label{fsfdkijvoirjovwevewvervwv}
\dots \to E^{K}(\partial Z)\to E^{G,K}(X,Z)\stackrel{s}{\to} E^{G}(X)\stackrel{r }{\to} \Sigma E^{K}(\partial Z)\to \dots
\end{equation}
\end{lem}
\begin{proof}
Since $(Z,\{X\setminus Z\})$ is a $K$-invariant complementary pair on $\Res^{G}_{K}(X)$, by the excision axiom for $E^{K}$  \cite[Def.\ 3.10 (2)]{equicoarse}, we have 
{a cartesian square}
\begin{equation*}
\begin{tikzcd}
E^K(\partial Z) \ar[r] \ar[d]& E^K(Z) \ar[d, "i"]\\
E^K(\{X\setminus Z\}) \ar[r, "j"'] & E^K(\Res^G_K(X)),
\end{tikzcd}
\end{equation*}
leading to
\begin{equation} \label{FibreSequenceZXZ}
  E^K(\partial Z) \to E^{K}(Z)\oplus E^{K}(\{X\setminus Z\}) \to E^K(\Res^G_K(X)) \stackrel{\delta}{\to} \Sigma E^K(\partial Z)
\end{equation}
in $\bC$. Equivalently, the right square in the following diagram is  cartesian:
\begin{equation}\label{vwjjvweiofvw}
\begin{tikzcd}[column sep=1.4cm]
E^{G,K}(X,Z)\ar[r, "{e_{Z}\oplus e_{X\setminus Z}}"] \ar[d, "s"]  & E^{K}(Z)\oplus E^{K}(\{X\setminus Z\}) \ar[r]\ar[d, "i+j"] & 0\ar[d] \\
E^{G}(X)\ar[r, "{c^{G}_{K}}"] & E^{K}(\Res^{G}_{K}(X))\ar[r, "\delta"] & \Sigma E^{K}(\partial Z) .
\end{tikzcd}
\end{equation}
Since the left square of \eqref{vwjjvweiofvw} is cartesian by  \cref{wfeioweedsfd} we conclude that the outer square is   cartesian. 
This is the desired fibre sequence \eqref{fsfdkijvoirjovwevewvervwv}, {defining $r$ as the composition of the two bottom horizontal maps.}
\end{proof}

\begin{ddd}\label{wegoijwegerwregwg}
We define the obstruction morphism as the composition \begin{equation}\label{eqwfpokqpowefqewfqwefewf}
r \simeq \delta  \circ {c^{G}_{K}}:E^{G}(X)\to \Sigma E^{K}(\partial Z)\ . \qedhere
\end{equation} 
\end{ddd}


Let $f:X^{\prime}\to X$ be a morphism of $G$-bornological coarse spaces and define the $K$-invariant subset $Z':=f^{-1}(Z)$ of $X'$.

\begin{lem}\label{twiojwtohwthtwhtgwerg}
There  exists a morphism 
\begin{equation*}
\tilde f_{*} : E^{G,K}(X^{\prime},Z^{\prime}) \longrightarrow E^{G,K}(X,Z)
\end{equation*}
fitting into a morphism of fibre sequences
\begin{equation}\label{evfvsdfvfdvs}
\begin{tikzcd}[column sep = 0.7cm]
\dots \ar[r] & E^{K}(\partial Z^{\prime})\ar[r]\ar[d, "\partial f_{*}"] & E^{G,K}(X^{\prime},Z^{\prime}) \ar[r, "s^{\prime}"]\ar[d, "\tilde f_{*}"] &E ^{G}(X^{\prime})\ar[d, "f_{*}"] \ar[r, "r^{\prime}"] &\Sigma E^{K}(\partial Z^{\prime})\ar[r]\ar[d, "\partial f_{*}"] &\dots\\ \dots\ar[r]&E^{K}(\partial Z)\ar[r] &E^{G,K}(X,Z)\ar[r, "s"]&E^{G}(X)\ar[r, "r"]&\Sigma E^{K}(\partial Z)\ar[r]&\dots
\end{tikzcd}
\end{equation}
\end{lem}

\begin{proof}
%
 First of all, the map $f$ induces morphisms $f_{*}:E^{G}(X^{\prime})\to E^{G}(X)$ and $\Res^{G}_{K}(f)_{*}:E^{K}(\Res^{G}_{K}(X^{\prime}))\to E^{K}(\Res^{G}_{K}(X))$, and the square
\begin{equation}\label{qeoirgerwfgrefwe}
\begin{tikzcd}
  E^G(X^\prime) \ar[d, "f_*"'] \ar[r, "{c^{G}_{K }}"] & E^K(\Res^G_K(X^\prime)) \ar[d, "\Res^G_K(f)_*"] \\
  E^G(X) \ar[r, "{c^{G}_{K}}"] & E^K(\Res^G_K(X))\end{tikzcd}
\end{equation}
commutes by naturality of the transfer.

Moreover, since $f(X'\setminus Z')\subseteq X\setminus Z$,  for every $U^{\prime}$ in $\cC_{X^{\prime}}^{G}$ we have $f( U^\prime [X^{\prime}\setminus Z^{\prime}])\subseteq   f(U^{\prime})[X\setminus Z]$
and $f( Z^\prime \cap U^\prime[X^{\prime}\setminus Z^{\prime}])\subseteq  Z\cap f(U^{\prime})[X\setminus Z]$. Therefore $f$ induces   morphism{s} of big families
$f_{X'\setminus Z'}:\{X^{\prime}\setminus Z^{\prime}\}\to \{X\setminus Z\}$ and $\partial f:\partial Z^{\prime}\to \partial Z$, and the  diagram  
\begin{equation}\label{weiorgwrfrefrewfrf}
\xymatrix{
  E^{K}(Z')\oplus E^{K}(\{X'\setminus Z'\}) \ar[r]^-{i'+j'}\ar[d]^{f_{Z',*}\oplus f_{X'\setminus Z',*}}&E^K(\Res^G_K(X'))\ar[d]^{\Res^{G}_{K}(f)_{*}} \ar[r]^{\delta}&\Sigma E^{K}(\partial Z')\ar[d]^{  \partial f_*}\\
 E^{K}(Z) \oplus E^{K}(\{X\setminus Z\})\ar[r]^{i+j}&E^K(\Res^G_K(X)) \ar[r]^{\delta} &\Sigma E^{K}(\partial Z) } 
\end{equation}
commutes by the naturality of the Mayer-Vietoris fibre sequence.
The  horizontal concatenation of the square in \eqref{qeoirgerwfgrefwe} and the right square in  \eqref{weiorgwrfrefrewfrf} yields the rightmost commutative square in \eqref{evfvsdfvfdvs}, which extends to the desired morphism \eqref{evfvsdfvfdvs} of fibre sequences.
%
\end{proof}

  \begin{lem}\label{efiojwerogergwreg}
 If  $\partial Z'$   consists of empty sets,
 then the composition $r\circ f_{*}:E^{G}(X^\prime)\to E^{G}(X)\to E^{K}(\partial Z)$ vanishes.
 \end{lem}
 
\begin{proof}
  Since $\partial Z'$  consists of empty sets, {we have $E^K(\partial Z') \simeq 0$, hence} $r^{\prime}\simeq 0$. The commutativity of the right square in \eqref{evfvsdfvfdvs} implies the assertion.
 \end{proof}
 
 \begin{ex}\label{weigjieqrogegwegrewg}
 If  $X'$ is obtained from $X$ by replacing the coarse structure by a smaller one, then the identity map is a morphism $X'\to X$.    
 In particular, we can take $X^\prime = X_{\mathrm{disc}}$, the bornological coarse space with the same underlying space and bornology as $X$ but the minimal coarse structure, generated by the empty set. This means that all entourages $U$ of $X^\prime$ are subsets of the diagonal, so that $U[X\setminus Z] \subseteq X \setminus Z$. This implies $Z \cap U[X\setminus Z] = \emptyset$ for each entourage $U$ of $X_{\mathrm{disc}}$, so $\partial Z = \{\emptyset\}$, the condition of  \cref{efiojwerogergwreg}.
\end{ex}

\begin{ex}\label{egiweogergewrgwergwerg}The following idea is inspired by  \cite[Sec. 3.2]{Hanke:aa}. 
Let {$a:Y\to Z$ be the embedding of a $K$-invariant subset} of $Z$.
Following \cite[Def. 3.9]{Hanke:aa} we could call $Y$ coarsely $E^{K}$-negligible if the map $E^{K}(a):E^{K}(Y)\to E^{K}(Z)$ is zero. The following more restrictive notion captures a reason for this triviality.

\begin{ddd}
$Y$ admits a {flasque exit} in $Z$ if the 
 embedding $a:Y\to Z$ extends to a morphism    $h:[0,\infty)\otimes Y\to Z$ in $K\BC$ such that $h_{|\{0\}\times Y}=a$.
\end{ddd}
 
  For example, if $K=\{1\}$,  $Y$ is $U$-bounded for some entourage $U$ of $X$ and there exists a coarse ray in $Z$ (a map from $[0,\infty)$ to $Z$) starting in a point of $Y$, then $Y$ admits a flasque exit in $Z$ \cite[Prop. 3.10]{Hanke:aa}.

 If $Y$ admits a flasque exit in $Z$, then the inclusion induces a zero map in coarse homology,
  $$0\simeq E^{K}(a):E^{K}(Y)\to E^{K}(Z)\ .$$
This follows from the fact that vanishing on flasque spaces is one of the defining properties of an equivariant coarse homology theory \cite[Definition~3.10]{equicoarse}.
 Indeed, ${E^{K}(a)}$ factorizes as \smash{$E^{K}(Y)\to E^{K}([0,\infty)\otimes Y)\stackrel{h_{*}}{\to} E^{K}(Z)$}, and $ E^{K}([0,\infty)\otimes Y)\simeq 0$ by flasqueness of {$[0,\infty)\otimes Y$}. 
 So the choice of $h$ provides a preferred choice of an equivalence ${0\simeq E^{K}(a)}$.
 
 \begin{ddd}\label{thiowrhrtehetrheth}
 We say that  the big family $\partial Z$ admits a  {two-sided flasque exit}, if for every $U$ in $\cC^{K}_{X}$ the  member $Z\cap U[X\setminus Z]$ of $\partial Z$ admits a flasque exit in  {both} $Z$ and   $U[X\setminus Z]$.
 \end{ddd}

If $ \partial Z$ admits a two-sided  flasque  exit, then the fibre sequence \eqref{FibreSequenceZXZ}  splits. 
We therefore get  
 a decomposition
\begin{equation*}
E^{K}(\Res^{G}_{K}(X))\simeq E^{K}(Z)\oplus E^{K}(\{X\setminus Z\})\oplus \Sigma E^{K}(\partial Z),
\end{equation*}
 where $\delta$ corresponds to the projection onto the last summand. 
 \end{ex}

\begin{ex}\label{gijergoiewrgwergwegwreg}
In this example we describe {a} situation where $\partial Z$ has a two-sided flasque exit.
Let $m,n$ be in $\nat$ and
consider the group $G:=\Z^{m+n}$. Then 
 $X:=\R^{m+n}$ is a   $\Z^{m+n}$-bornological coarse space with the structures induced by the standard metric with the usual action by translations. 
  We let  $K:=\Z^{m}$ be the first summand of $\Z^{m+n}\cong \Z^{m}\oplus \Z^{n}$. We furthermore take 
  the $\Z^{m}$-invariant subset
 \[Z:=\R^{m+n}_{+} :=\{ (x,y)\in \R^{m+n}\:|\:y\in [0,\infty)^{n}\}\ ,\]
 where identify $\R^{m+n}\cong \R^{m}\times \R^{n}$ in order to write its elements as pairs.
Note that $\partial Z$ admits a two-sided flasque exit (see \cref{thiowrhrtehetrheth}).

 But actually in this case the situation is even simpler.
Namely,   $Z$ and all members of $\{X\setminus Z\}$ are flasque.   This implies that the boundary map $\delta$ in the cofibre sequence \eqref{FibreSequenceZXZ} is an equivalence. The obstruction morphism  \begin{equation*}
 r :  E^{\Z^{m+n}}(\R^{m+n})   \longrightarrow   {\Sigma} E^{\Z^{m}}(\partial Z)
  \end{equation*}
  in the fiber sequence \eqref{fsfdkijvoirjovwevewvervwv} is therefore equivalent to the transfer morphism 
  \begin{equation*}
{ c^{\Z^{m+n}}_{\Z^{m}}} :  E^{\Z^{m+n}}(\R^{m+n}) \to   E^{\Z^{m}}(\Res^{\Z^{m+n}}_{\Z^{m}}(\R^{m+n})).\qedhere
  \end{equation*} 
 \end{ex}

    \section{Topological coarse $K$-homology}\label{thrertherhertheth}

If $\bC$ is a  $C^{*}$-category with {strict} $G$-action {admitting all orthogonal AV-sums \cite[Def. 7.1]{cank},  then we  have the equivariant coarse $K$-homology functor $$K\cX_{\bC}^{G}:G\BC\to \Sp$$ with transfers \cite[Def. 6.2.1 \& Thm. 6.3 \& 9.7]{coarsek} (with $K$-theory functor for $C^{*}$-categories  \cite[Def. 14.3]{cank} in place of $\Homol$).}
If we specialize this construction to the category {$\bC=\Hilb_{c}(\C)$ of Hilbert spaces and compact operators} with the trivial $G$-action, then the resulting functor will be denoted by $ K\cX^{G}$.    For simplicity, in the present paper we will stick to this example.

On very proper $G$-bornological  coarse spaces represented by locally compact metric spaces with isometric $G$-action this is the usual equivariant coarse $K$-theory. Its value on such a space   
  is the $K$-theory of the equivariant Roe algebra associated to this space  \cite[Thm. 6.1]{indexclass}, see also \eqref{bsgfokwtpobsgbsdfbdfsbfdsb}. For $\bC=\Hilb_{c}(A)$ the $C^{*}$-category of
  Hilbert $C^{*}$-modules of a $C^{*}$-algebra $A$ with $G$-action and compact operators the functor
 $K\cX_{\bC}^{G}$ is related {to} the topological equivariant $K$-homology  functor with coefficients in $A$ $$K^{G}_{A}:G\Orb\to \Sp$$ constructed by \cite{davis_lueck} (see also \cite{kranz}) 
appearing in the Baum-Connes conjecture  by
$$K^{G}_{A}(S)\simeq K\cX^{G}_{\bC}(S_{min,max}\otimes G_{can,min})\ , \quad S\in G\Orb\ .$$ 
Here $S_{min,max}$ is the $G$-bornological coarse space given by the  transitive $G$-set $S$ equipped with the minimal coarse structure and the maximal bornology. For the trivial group the coarse homology theory  $K\cX$ constitutes the  target of the
coarse assembly map featuring the  coarse Baum-Connes conjecture and should not be confused with the domain of this assembly map which is the coarsification of the locally finite   $K$-homology theory represented by the spectrum $KU$ \cite[Sec. 7]{buen} and \cite{ass}.

Note that by \cite[Thm. 6.3]{coarsek} the family $(K\cX^{G})_{G}$ of equivariant coarse homology theories has induction equivalences in the sense explained in \cref{wfrrefwerfergwrgwergwefref}.

We  now apply the constructions described in \cref{wrthwoiejgioergrewg} to the equivariant coarse homology theory with transfers $K\cX^{G}$. 

 In particular, for a subgroup $K$ of $G$ and a $K$-invariant subset $Z$ of $X$ we can construct the spectrum $K\cX^{G,K}(X,Z)$ as in \cref{wfeioweedsfd} and obtain the fibre sequence \eqref{fsfdkijvoirjovwevewvervwv}:
 \begin{equation}\label{fsfdkijvoirjovwevewvervwv12}
\dots \to K\cX^{K}(\partial Z)\to K\cX^{G,K}(X,Z)\to K\cX^{G}(X)\stackrel{r }{\to} \Sigma K\cX^{K}(\partial Z)\to \dots\ .
\end{equation}

Under certain restrictions on $G$ and $X$ the values of  
$K\cX^{G}(X)$ 
 can be described in terms of the topological $K$-theory of  Roe algebras. The main result of the present subsection is a presentation of $K\cX^{G,K}(X,Z)$  and
the sequence  
 \eqref{fsfdkijvoirjovwevewvervwv12}  in terms of $C^{*}$-algebras. This interpretation will be used for application to the spectral theory of differential operators in \cref{erguiqrgeqfweqef}.

We assume that $G$ is countable. 
We furthermore assume that the $G$-bornological coarse space  $X$ is very proper  \cite[Def.\ 3.7]{indexclass}. 
Very properness  is a technical condition which is e.g.\ satisfied for  $G$-bornological coarse space represented by complete Riemannian manifolds with a proper  isometric $G$-action. 
It ensures by  \cite[Prop. 4.2]{indexclass} the existence of  ample equivariant $X$-controlled Hilbert space $(H,\phi)$ \cite[Def.\ 4.1]{indexclass}.
%
Here $H$ is a complex Hilbert space  with $G$-action and  $\phi$ is an invariant, finitely additive projection-valued measure on $X$ defined on all subsets. For any such equivariant $X$-controlled Hilbert space
 one can define the {Roe algebra}
$C(X,H,\phi)$ {(see Definitions 2.3, 3.6 \& 3.8 of \cite{indexclass})}.

By \cite[Thm. 6.1]{indexclass} we have a canonical (up to equivalence) morphism of spectra \begin{equation}\label{bsgfokwtpobsgbsdfbdfsbfdsb}
K(C(X,H,\phi))\stackrel{\kappa_{(X,H,\phi)}}{\simeq} K\cX^{G}(X)\ .
\end{equation}

  Note that {the assumption that $X$ is very proper as a $G$-bornological coarse space implies that $\Res^{G}_{K}(X)$ is a very proper $K$-bornological coarse space}. If we assume that $Z$ is a {nice subset} of $\Res^{G}_{K}(X)$ for some $K$-invariant  entourage of $X$ \cite[Def.\ 8.2]{indexclass}, then
$Z$ is a very proper $K$-bornological coarse space \cite[Prop. 8.3]{indexclass}.

We consider the subspace $H_{Z}:=\phi(Z)(H)$ of $H$,
and we  let $\phi_{Z}$ denote the  restriction of  the projection-valued measure $\phi$ to subsets of $Z$ and $H_{Z}$. Then $(H_{Z},\phi_{Z})$ is a $K$-equivariant $Z$-controlled Hilbert space. It is easy to check that it is ample.

{The {localized Roe algebra} is the subalgebra} $C(\partial Z,H_{Z},\phi_{Z})$ of $C(Z,H_{Z},\phi_{Z})$ generated by operators supported on the members of the big family $\partial Z$ \eqref{feorihioweregw}. It is a closed ideal \cite[Lem. 7.57]{buen}.  We define a linear map
$$q:C(X,H,\phi)\to C(Z,H_{Z},\phi_{Z})\ , \quad q(A):=\phi(Z)A\phi(Z)\ .$$
{This map is $*$-preserving, but} in general $q$ is not a homomorphism of $C^{*}$-algebras.

 \begin{ddd}
We define the $C^{*}$-algebra $C^{G,K}(X,Z)$ as the sub-$C^{*}$-algebra of $C(Z,H_{Z},\phi_{Z})$ generated by 
$C(\partial Z,H_{Z},\phi_{Z})$ and $\im(q) $.
\end{ddd}

{Recall that for an entourage $U$ of $X$, a subset $B$  of $X$ is called $U$-bounded, if $B\times B\subseteq U$ \cite[Def.~2.15]{buen}. }
 
\begin{ddd}\label{9wtgiohwgregwerg}
We call $Z$ {wide} if for every entourage $U$ of $X$ and each $U$-bounded subset $Y$ of $X$, there exists $g$ in $G$ such that $gY\subseteq Z\setminus U[X\setminus Z]$.
\end{ddd}

\begin{theorem}\label{qwefihfuiewqfeefewfewfeqwfqfeceq} If  $Z$ is wide, then
we have a short exact sequence of $C^{*}$-algebras
 \begin{equation}\label{qewfoifjoqwefqwefqewf}
0\to C(\partial Z,H_{Z},\phi_{Z})\stackrel{\iota}{\to} C^{G,K}(X,Z)\stackrel{\sigma}{\to} C(X,H,\phi)\to 0 \ ,
\end{equation}
where the map $\iota$ is the natural inclusion. 
\end{theorem}

\begin{proof}
Since $C(\partial Z,H_{Z},\phi_{Z})$ is a closed ideal in $C(Z,H_{Z},\phi_{Z})$ it is also a closed ideal in $C^{G,K}(X,Z)$.  

{We need to construct the map $\sigma$.} As in the proof of \cite[Lemma 7.58]{buen}, {one shows}  that   for $A, A^{\prime}$ in $ C(X,H,\phi)$ we have 
\begin{equation} \label{ContainedInIdeal}
q(A) q(A^{\prime})-  q(AA^{\prime})\in C(\partial Z,H_{Z},\phi_{Z})\ .
\end{equation}
Since $q$ is $*$-preserving, it then follows that $C(\partial Z,H_{Z},\phi_{Z})+\im(q)$ is a $*$-subalgebra of $C(Z,H_{Z},\phi_{Z})$, and that the composition
 \begin{equation*}
 \bar q:C(X,H,\phi)\stackrel{q}{\to} C^{G,K}(X,Z) \to C^{G,K}(X,Z)/C(\partial Z,H_{Z},\phi_{Z})
 \end{equation*}
is a homomorphism of $C^{*}$-algebras. 
We claim that $\bar q$ is injective. 
Assuming the claim, since $\bar q$ has dense range by construction, it is an isomorphism. Hence we can define $\sigma$ as the composition
\begin{equation*}
\sigma : C^{G,K}(X,Z) \to C^{G,K}(X,Z)/C(\partial Z,H_{Z},\phi_{Z})\stackrel{\bar q^{-1}}{\to} C(X,H,\phi)
\end{equation*}
and obtain the desired exact sequence \eqref{qewfoifjoqwefqwefqewf}.

We now show the claim. 
Let $A$ be in $C(X,H,\phi)$  and assume that $\bar q(A)=0$, i.e., that 
\begin{equation*}
q(A) = \phi(Z)A \phi(Z)\in C(\partial Z,H_{Z},\phi_{Z})\ .
\end{equation*}
Then
for every $\epsilon$ in $(0,\infty)$ we can then choose 
 an entourage $U$ and a $U$-bounded subset $Y$ of $X$ such that the following two inequalities are satisfied: $$\|\phi(Y)A\phi(Y)\|\ge \frac{1}{2}\|A\| \ , \quad \|\phi(Z\setminus U[X\setminus Z]) A \phi(Z\setminus U[X\setminus Z])\|<\epsilon\ .$$   Since $Z$ is wide we can choose an element $g$ in $G$ such that $gY\subseteq Z\setminus U[X\setminus Z]$. Then using  the equality $$
\phi(Z\setminus U[X\setminus Z])\phi(gY)=\phi((Z\setminus U[X\setminus Z])\cap gY)=\phi(gY)\ ,$$ the equivariance $\phi(Y)=g\phi(gY)g^{-1} $ of $\phi$ and the equivariance $ g^{-1}Ag=A$ of $A$ we get
\begin{multline}\phi(Y)A\phi(Y)= g\phi(gY)g^{-1}Ag\phi(gY)g^{-1}
=g\phi(gY) A \phi(gY)g^{-1}\\=g\phi(gY) \phi(Z\setminus U[X\setminus Z])A \phi(Z\setminus U[X\setminus Z])\phi(gY)g^{-1}\end{multline}
Hence $$\frac{1}{2}\|A\|\le \|\phi(Y)A\phi(Y)\| \le \epsilon\ .$$
Since $\epsilon$ is arbitrary this implies that $A=0$.
\end{proof}

The sequence of $C^{*}$-algebras \eqref{qewfoifjoqwefqwefqewf} induces
a fibre sequence of $K$-theory spectra
 \begin{equation}\label{vgqvqwecqwecwec}
 \begin{tikzcd}[column sep=0.5cm]       
\cdots \arrow{r}  \ar[dr, phantom, "\scriptstyle{\delta^{C^*}}", near start, bend right = 40]&
K(C(\partial Z,H_{Z},\phi_{Z})) \ar[r] &
K(C^{G,K}(X,Z)) \ar[r] \arrow[phantom, ""{coordinate, name=Z}]{d} &
K(C(X,H,\phi)) 
  \arrow[
    rounded corners,
    to path={
      -- ([xshift=2ex]\tikztostart.east)
      |- (Z) [near end]\tikztonodes
      -| ([xshift=-2ex]\tikztotarget.west)
      -- (\tikztotarget)
    }
  ]{dll} \\
 &
\Sigma K(C(\partial Z,H_{Z},\phi_{Z})) \arrow{r} &
\cdots \ .&
\end{tikzcd}
\end{equation}
The following is the main theorem of this section.

\begin{theorem}\label{eiogwegergwegergwerg}
The  fibre sequences \eqref{vgqvqwecqwecwec}
and \eqref{fsfdkijvoirjovwevewvervwv12}
are canonically equivalent.
\end{theorem}

\begin{proof}
We use the comparison maps $\kappa_{\bullet}$ as in \eqref{bsgfokwtpobsgbsdfbdfsbfdsb} and consider the following diagram:
\begin{equation}\label{eoirgjowergregwergerg}
\xymatrix{K(C^{G,K}(X,H,\phi))\ar[r]^{\sigma_{*}}\ar@{..>}[d]_{\simeq}&K(C(X,H,\phi))\ar[r]^{\delta^{C^{*}}}\ar[d]_{\simeq}^{\kappa_{(X,H,\phi)}}&\Sigma K(C(\partial Z,H_{Z},\phi_{Z}))\ar[d]_{\simeq}^{\kappa_{(\partial Z,H_{Z},\phi_{Z})}}\\K\cX^{G,K}(X,Z)\ar[r]^{s}&K\cX^{G}(X)\ar[r]^{r }&\Sigma K\cX^{K}(\partial Z)}\ .
\end{equation}
It suffices to show that the right square commutes. Then we use  the universal properties of the fibre sequences in order
to  obtain the dotted arrow  and the fact that it is an equivalence.

We expand the right square as follows:
\begin{equation}\label{vewvewrvervevervwevervrev}
\xymatrix{K(C(X,H,\phi))\ar[r]^{ K(i)}\ar@/^1cm/[rr]^{\delta^{C^{*}}}\ar[d]_{\simeq}^{\kappa_{(X,H,\phi)}} & K(C(\Res^{G}_{K}(X,H,\phi)))\ar[d]_{\simeq}^{\kappa_{\Res^{G}_{K}(X,H,\phi)}}\ar[r]^{\delta^{MV}} & \Sigma K(C(\partial Z,H_{Z},\phi_{Z}))\ar[d]_{\simeq}^{\kappa_{(\partial Z,H_{Z},\phi_{Z})}} \\
K\cX^{G}(X)\ar[r]^{{c^{G}_{K}}}  
 & K\cX^{K}(\Res^{G}_{K}(X))\ar[r]^{\delta } & {\Sigma}K\cX^{K}(\partial Z)}
\end{equation} 
 
In order to see that the right square and the upper triangle commute, we consider 
the following morphisms of short exact sequences of $C^{*}$-algebras \begin{equation}\label{asdvkjnkqfv}
\begin{tikzcd}[column sep=0.3cm] 
0\ar[r] &C(\partial Z,H_{Z},\phi_{Z}) \ar[d, equal] \ar[r, "\iota"] & C^{G,K}(X,Z)\ar[d] \ar[r, "\sigma"] & \ar[ddl, "i",dotted, near end, bend left=5]C(X,H,\phi) \ar[d, "q"] \ar[r] & 0\\
0\ar[r ]& C(\partial Z,H_{Z},\phi_{Z}) \ar[d] \ar[r] & C( Z,H_{Z},\phi_{Z})  \ar[r, "p", near start] \ar[d]& \frac{C( Z,H_{Z},\phi_{Z})}{C(\partial Z,H_{Z},\phi_{Z})} \ar[r] \ar[d, "b", "\cong"'] &0\\
0\ar[r]&C(\{X\setminus Z\},H_{X\setminus Z},\phi_{X\setminus Z}) \ar[r]& C(\Res^{G}_{K}(X,H,\phi)) \ar[r, "a"] &\frac{ C(\Res^{G}_{K}(X,H,\phi)) }{C(\{X\setminus Z\},H_{X\setminus Z},\phi_{X\setminus Z})}\ar[r]&0
\end{tikzcd}
\end{equation}
 The vertical maps are induced by the natural inclusions.
 
Applying the $K$-theory functor for $C^{*}$-algebras to this diagram {and extending to the right}, we get a morphism of  fibre sequences
 \begin{equation}\label{qwefewfqwefqewfqewf}
 \begin{tikzcd}[column sep=0.5cm] 
 K(C^{G, K}(X, Z) \ar[d] \ar[r, "K(\sigma)"] &K( C(X, H, \phi) )\ar[r, "\delta^{C^{*}}"] \ar[ddl, dotted, bend left=5, near start, "K(i)"'] \ar[d, "K(q)"] &\Sigma K(C(\partial Z,H_{Z},\phi_{Z}))\ar[d, equal] \\
K(C(Z,H_{Z},\phi_{Z})\ar[d]\ar[r, "K(p)"] &K( \frac{C( Z,H_{Z},\phi_{Z})}{C(\partial Z,H_{Z},\phi_{Z})})\ar[r, "\delta^{C^{*}\prime}"] \ar[d, "K(b)", "\simeq"'] &\Sigma K(C(\partial Z,H_{Z},\phi_{Z}))\ar[d] \\
K( C(\Res^{G}_{K}(X,H,\phi)))\ar[r, "K(a)"] \ar[rru, dotted, "\delta^{MV}", near start, bend left=5] & K\bigl(\frac{ C(\Res^{G}_{K}(X,H,\phi))}{C(\{X\setminus Z\},H_{X\setminus Z},\phi_{X\setminus Z})}\bigr)\ar[r, "\delta^{C^{*}\prime\prime}"] & \Sigma K(C(\{X\setminus Z\},H_{X\setminus Z},\phi_{X\setminus Z}) ) \ .
\end{tikzcd}
\end{equation}
By definition, the Mayer-Vietoris boundary map for the complementary pair $(Z, \{X \setminus Z\})$ on $X$ is  given by  
\begin{equation}\label{ervwrekpwervewveverv}
\delta^{MV}:=\delta^{C^{*}\prime}\circ {K(b)^{-1}}\circ K(a):K(C(\Res^{G}_{K}(X,H,\phi)))\to \Sigma K(C(\partial Z,H_{Z},\phi_{Z}))\ .
\end{equation} 
{Similarly, the boundary map $\delta$ is obtained} from the diagram
 \begin{equation}\label{qwefewfqwefqewfqewf1}
\begin{tikzcd} K\cX^{K}(Z )\ar[d]\ar[r] &\frac{K\cX^{K}( Z )}{K\cX^{H}( \partial Z)}\ar[r, "\delta^{\prime}"] \ar[d, "{b^\prime}"', "\simeq"] & \Sigma K\cX^{K}(\partial Z )\ar[d] \\
K\cX^{K}( \Res^{G}_{K}(X ))\ar[r, "a'"] \ar[rru, dotted, "\delta", near start, bend left=5]  & \frac{K\cX^{K}( \Res^{G}_{K}(X)) }{K\cX^{K}(\{X\setminus Z\}) } \ar[r, "\delta^{\prime\prime}"] &\Sigma K\cX^{K}(\{X\setminus Z\} )
\end{tikzcd}
\end{equation}
as $$\delta:=\delta'\circ (b^\prime)^{-1}\circ a':K\cX^{K}( \Res^{G}_{K}(X ))\to {\Sigma}K\cX^{H}( \partial Z)\ .$$
Here we write cofibres as quotients.
In view of  \cite[Thm. 6.1]{indexclass} the transformation $\kappa_{\bullet}$ induces a morphism from the {lower two rows of} the diagram \eqref{qwefewfqwefqewfqewf}
to the diagram \eqref{qwefewfqwefqewfqewf1}.  {Now the right square of \eqref{vewvewrvervevervwevervrev} is contained in the commutative cube obtained this way and hence commutative}. 
  
 The boundary map $\delta^{C^{*}}$ is induced by the first sequence in \eqref{asdvkjnkqfv}.
The naturality of the boundary maps for the map from the first to the second sequence  in \eqref{asdvkjnkqfv} and the existence of the dotted lift therefore imply that the upper triangle in \eqref{vewvewrvervevervwevervrev} commutes. 

Finally, the  left  square in \eqref{vewvewrvervevervwevervrev} commutes in view of   \cite[Rem. 10.6]{coarsek}.
   \end{proof}

 \section{Operators in Roe algebras}

  Assume that $A$ is a selfadjoint operator in some $C^{*}$-algebra $\cA$  with spectrum $\sigma(A)$, and let  $I$  be an interval in $\R$. 

\begin{ddd}\label{weroigjoergergegergwegergg}
We say that $I$ is a spectral interval of $A$ if    $\partial I\cap \sigma(A)=\emptyset$.
\end{ddd}

If $I$ is a spectral interval for $A$, then the spectral projection $E_{A}(I)$ also belongs to $\cA$, where $E_A$ is the spectral measure of $A$
{and hence defines}  
a class $[E_{A}(I)]$ in $K_{0}(\cA)$.
The point here is that under the assumption of \cref{weroigjoergergegergwegergg}, we can write $E_A(I)$ as a continuous function of $A$ instead of using the characteristic function.

 As in \cref{thrertherhertheth} we consider a countable group $G$ and  let $X$ be a very proper $G$-bornological coarse space.
  We choose  an   $X$-controlled ample Hilbert space $(H,\phi)$. We furthermore consider a subgroup $K$ and a wide $K$-invariant subset $Z$
  of $X$ (see \cref{9wtgiohwgregwerg}).
   
     We consider operators $A$ in $C(X,H,\phi)$ and $B$ in $C(Z,H_{Z},\phi_{Z})$. In the following we use the maps {$$p:C(Z,H_{Z},\phi_{Z})\to \frac{C(Z,H_{U},\phi_{Z})}{C(\partial Z, H_{Z},\phi_{Z})}\ , \quad  q:C(X,H,\phi)\to  \frac{C(Z,H_{U},\phi_{Z})}{C(\partial Z, H_{Z},\phi_{Z})}$$ from}   the diagram \eqref{asdvkjnkqfv}.
     
  \begin{ddd}
  We say that $A$ and $B$ are affiliated if $q(A)=p(B)$.
    \end{ddd}
 
 Using $\kappa_{(X,H,\phi)}$ in  \eqref{bsgfokwtpobsgbsdfbdfsbfdsb} we consider $[E_{A}(I)]$ as a class in $K\cX^{G}(X)$.
 {In order to shorten the notation we will drop the identification \eqref{bsgfokwtpobsgbsdfbdfsbfdsb} from the notation. Recall the  obstruction morphism  $r:K\cX^{G}(X)\to K\cX^{K}(\partial Z)$ from \eqref{fsfdkijvoirjovwevewvervwv12}.}
\begin{prop}\label{ajfiovafdsvavdsvadv}
Assume:
\begin{enumerate}
\item $A$ and $B$ are affiliated.
\item $I$ is a spectral inverval for {both} $A$ and $B$.
 \end{enumerate}
Then $r ([E_{A}(I)])=0$.
\end{prop}
\begin{proof}
We have  the equality \begin{equation}\label{f32f23f3f23f3}
p(E_{B}(I))=E_{p(B)}(I)=E_{q(A)}(I)=q(E_{A}(I))\ .
\end{equation}  {Recall  the boundary map $\delta^{C^{*}\prime}$ in \eqref{qwefewfqwefqewfqewf}. 
We calculate \begin{eqnarray*}\label{qfeqoijoqiwefewfqewf}
r ([E_{A}(I)])& \stackrel{\eqref{eoirgjowergregwergerg}}{=}&   \delta^{C^{*}}([E_{A}(I)])  \stackrel{\eqref{vewvewrvervevervwevervrev}}{=} \delta^{MV}(K(i)([E_{A}(I)])) \\&
\stackrel{\eqref{ervwrekpwervewveverv}}{=} &
\delta^{C^{*}\prime}(K(b^{-1})(K(a)(K(i)([E_{A}(I)]))) )
 \stackrel{\eqref{asdvkjnkqfv}}{=}\delta^{C^{*}\prime}( K(q)([E_{A}(I)]))\\&\stackrel{\eqref{f32f23f3f23f3}}{=}  &
\delta^{C^{*}\prime} (K(p)([E_{B}(I)]))=
 0\ .
\end{eqnarray*}
The last equivalence follows from
 $\delta^{C^{*}\prime}\circ K(p)  \simeq 0$ since the upper horizontal part in \eqref{qwefewfqwefqewfqewf} is a segment of a fibre sequence.}
 \end{proof}

  \cref{ajfiovafdsvavdsvadv} gives a $K$-theoretic  condition for checking that certain elements  in the resolvent set of $A$
belong to the spectrum of $B$.
More precisely we have the following corollary.
\begin{kor}
Assume:
\begin{enumerate}
\item $A$ and $B$ are affiliated.
\item $I$ is a spectral inverval for $A$. \item 
$r ([E_{A}(I)])\not= 0$.
 \end{enumerate}
Then  $\partial I\cap \sigma(B)\not= \emptyset$.
\end{kor}

 \section{Differential operators}\label{erguiqrgeqfweqef}

Let $M$ be a complete Riemannian manifold with an isometric action of $G$. Let furthermore $D$ be an elliptic selfadjoint differential operator on $M$ acting on sections of some {equivariant} Hermitian vector bundle $E\to M$.  We will consider the following two cases simultaneously.

\begin{ass}\label{rwejgoiejgoiergwregwreg}\mbox{} \begin{enumerate}
 \item (Dirac case:) $D$ is  a first order Dirac type operator of degree $n$. The degree is implemented by a grading and the action an appropriate   Clifford algebra. We will hide both structures   from the notation, see 
\cite[Sec. 7]{indexclass} for details. 
\item\label{rojopwe} (Laplace case): $D$ is a second order  Laplace type operator  {with non-negative spectrum}.  
\end{enumerate}
 \end{ass}
   
%

We consider $M$ as a bornological coarse space with the metric coarse structure and the minimal compatible bornology.  We consider the Hilbert space $H_{0}:=L^{2}(M,E)$ with the control $\phi_{0}$ as described in \cite[Sec. 9]{indexclass}. We can assume that $(H_{0},\phi_{0})$ embeds into an ample $M$-controlled Hilbert space $(H,\phi)$.  
Using this embedding we can and will view the Roe algebras associated to $(H_{0},\phi_{0})$ as subalgebras of the corresponding Roe algebras associated to $(H,\phi)$.

We now assume that $Z$ is {a} $K$-invariant 
{codimension-zero submanifold} of $M$ with smooth boundary\footnote{We add a superscript $\infty$ in order to distinguish this boundary from the coarse boundary of $Z$ introduced above.} $\partial^{\infty} Z$. 
We further assume that $D^{\prime}$ is a selfadjoint extension of the restriction of $D$ to $Z$
determined by imposing a  local  {$K$-invariant} elliptic boundary condition at $\partial^{\infty} Z$.
In detail this means in the Dirac case that $D^{\prime}$ is the closure of the unbounded operator  $(D_{|Z},\dom(D_{|Z}))$
on $L^{2}(Z,E_{|Z})$ with domain   {the} $K$-invariant subspace $$\dom(D_{|Z}):= \{\phi\in C^{\infty}_{c}(Z,E_{|Z})| (P\phi)_{|\partial Z^{\infty}}=0\}$$  determined  by a $K$-invariant section $P$ in $C^{\infty}(Z,\Hom(E,F))$ for some auxiliary $K$-equivariant vector bundle 
$F$ on $Z$. 
Thereby $P$ must be such that
$ (D_{|Z},\dom(D_{|Z}))$ is essentially selfadjoint and (for ellipticity) such that the Lopatinski-Shapiro condition is satisfied.
The Laplace operator case is similar, except that $P$ is now allowed to be a first order differential operator.

 In the Dirac case we assume that this boundary condition is in addition compatible with the grading and the Clifford action,  while in the Laplace case we assume that $D^{\prime}$ is still non-negative. 
 
In the Laplace case the typical examples are Dirichlet and Neumann boundary conditions, while  typical examples in the Dirac case are the absolute and relative boundary conditions for the Euler operator, i.e., the Dirac operator associated to the de Rham complex with the even/odd grading. 
Note that in the Dirac case such local elliptic selfadjoint boundary conditions do not always exist (e.g. for the $Spin^{c}$-Dirac operator) in contrast to the non-local Atiyah-Patodi-Singer boundary condition \cite{MR397797}.

Let $\varphi \in C_{0}(\R)$.
\begin{prop}\label{qoijfoqefeqewfwefeqwfe}
 {We} have  $\varphi(D)\in C(M,H,\phi)$ and  
 $\varphi(D^{\prime})\in C(Z,H_{Z},\phi_{Z})$.
 \end{prop}
 \begin{proof}
 For the Dirac case
 the assertion $\varphi(D)\in C(M,H_{0},\phi_{0})$ is \cite[Prop. 10.5.6]{higson_roe}.
The case of $D^{\prime}$ has the same proof taking into account that the wave equation still has finite propagation speed in view of the locality of the boundary condition. 

For the Laplace case, by positivity of $D$ and $D^{\prime}$ we can assume that $\varphi$ is even.
We then use the finite propagation speed of the wave operator family $\cos(t\sqrt{D})$ and argue similarly as in the Dirac case {(see e.g., Lemma~1.7 in \cite{LudewigThiangGapless})}.
  \end{proof}
  
\begin{prop}\label{eirogjwoegwergerwf} 
 {The} operators $\varphi(D)$ and $\varphi(D^{\prime})$ are affiliated.
\end{prop}

\begin{proof}
In the Laplace case we set $\varphi_{1}:=\varphi$, and in the Dirac case we use $\varphi_{1}(t):=\varphi(t^{2})$. 
It suffices to show the assertion for $\varphi_{1}$ in $C_{0}(\R)$ with $\hat \varphi_{1}\in C_{c}(\R)$. 
This implies the general case since these functions are dense in $C_{0}(\R)$.

 Let $U_{R}:=\{d\le R\}$ be the metric entourage of size $R$ on $M$.  
 If $\supp({\hat \varphi_{1}})\subseteq[-R,R]$, then $\phi(D)$ and $\phi(D')$ are  $U_{R}$-controlled and
   $$ \phi(Z)( \varphi(D)-\varphi(D'))\phi(Z\setminus U_{R}[X\setminus Z])=0\ .$$ It follows that 
$$\phi(Z)\varphi(D)\phi(Z)-\varphi(D')=\phi(Z)\varphi(D)\phi(Z\cap U_{R}[X\setminus Z]){-} \varphi(D') 
{\phi}(Z\cap U_{R}[X\setminus Z])\ .$$
Both summands on the right-hand side belong to $C(\partial Z,H_{Z},\phi_{Z})$.
Consequently $q(\varphi(D))=p(\varphi(D^{\prime}))$.
 \end{proof}

 Let $\sigma(D)$ denote the spectrum of $D$.
Let $\lambda$ be in $ \R\setminus \sigma(D)$.

\begin{prop} \label{ewgiowgregwergwrg}Assume that $D$ is of Laplace type.  
\begin{enumerate} 
\item \label{feuivhiqvcq}  We have
$E_{D}(-1,\lambda)\in C(M,H,\phi)$. 
 \item  \label{feuivhiqvcq1} If $r ([E_{D}(-1,\lambda)])\not=0$, then $\lambda\in \sigma(D^{\prime})$.
 \end{enumerate}
 
\end{prop}
\begin{proof}
The assertions are obvious if $\lambda$ is negative. In the following we  assume that $\lambda$ is non-negative.
We can choose $\varphi$ in $C_{0}(\R)$ such that  it {is} strictly monotoneously decreasing on $[-1,\infty)$. Since $D$ is positive, by the spectral mapping theorem  we have the equality $$E_{D}(-1,\lambda)=E_{\varphi(D)}((\varphi(\lambda),\varphi(-1))\ .$$ Since $\varphi(-1)$ and $\varphi(\lambda)$ are  not in the spectrum of $\varphi(D)$, the interval  $(\varphi(\lambda),\varphi(-1))$ is a spectral interval for $\varphi(D)$. Hence
$ E_{\varphi(D)}((\varphi(\lambda),\varphi(-1))\in  C(M,H,\phi)$. This implies (\ref{feuivhiqvcq}).
 
In order to show   (\ref{feuivhiqvcq1}) we assume by contradiction that $\lambda\not\in \sigma(D^{\prime})$. Then similarly as above 
$(\varphi(\lambda),\varphi(-1))$ is a spectral interval for 
 $\varphi(D^{\prime})$, too.  We conclude by  
\cref{ajfiovafdsvavdsvadv} that
$r_{X,Z}([E_{D}(-1,\lambda)])=0$ {which contradicts our Assumption (\ref{feuivhiqvcq1}).}
\end{proof}

If $D$ is  of Dirac type (of degree $n$), then  we have an index class
$\ind^{G}(D)$ in $K\cX_{n}^{G}(M)$. 
\begin{prop} \label{fergoijergoewgergewrgreg} We assume that $D$ is of Dirac type. If $D^{\prime}$ exists, then we
 have $r (\ind(D))=0$.
\end{prop}
\begin{proof}  
The index class $\ind^{G}(D)$ is represented by the homomorphism 
$$C_{0}(\R)\ni \varphi \mapsto \varphi(D)\in C(M,H,\phi)\ ,$$
where we use the homomorphism picture of $K$-theory, see  e.g. \cite[Sec. 4.1]{MR3551834}.
 Similarly,  if $D^{\prime}$ exists, then $\ind^{K}(D^{\prime})$ in $K\cX^{K}_{n}(Z)$ is represented by 
 the homomorphism 
 $$C_{0}(\R)\ni \varphi \mapsto \varphi(D')\in C(Z,H_{Z},\phi_{Z})\ .$$
  Since
$q(\varphi(D))=p(\varphi(D'))$ by  \cref{eirogjwoegwergerwf} we conclude that
$q(\ind^{G}(D))=p(\ind^{K}(D^{\prime}))$.
The  assertion now follows from the long exact sequence associated to the pair $(Z,\partial Z)$, see also {the proof of \cref{ajfiovafdsvavdsvadv}.}
 \end{proof}

\begin{kor}
If $r (\ind^{G}(D))\not=0$, then $D^{\prime}$ does not admit a {$K$-invariant,} local, selfadjoint, elliptic boundary condition which is compatible with the Clifford action.
\end{kor}

For $G$-equivariant Dirac operators $\Dirac$  we can calculate $r(\ind^{G}(\Dirac))$ as follows.
The choice of $Z$ induces an orientation and hence trivialization of the normal bundle of $\partial^{\infty} Z$ by the out-going normal. 
The Dirac operator $ \Dirac$ then naturally induces a $K$-invariant Dirac operator $\partial  \Dirac$ on $\partial^\infty Z$ which is well-defined up to zero order perturbations (if we would require a geometric product structure, then the operator would be canonical).
This restriction procedure works such that   if  $\Dirac$ is the Dirac operator on $M$ associated to a $Spin^{c}$-structure or $Spin$-structure, then
$\partial \Dirac$ is the Dirac operator associated to the induced $Spin^{c}$- or   $Spin$-structure
induced on $\partial^{\infty} Z$.

The  following result is well-known and an equivariant version of Roe's partitioned manifold index theorem, see  e.g.  \cite{Zeidler:2015aa}, \cite{Nitsche:2019aa}.  
 \begin{lem}\label{qroigjhqoegrgegwergregwergre}
We have
$r(\ind^{G}(\Dirac))=\ind^{K}(\partial \Dirac)$. 
\end{lem}
\begin{proof}
%
%
%

We will  sketch a version of  Zeidler's proof  \cite{Zeidler:2015aa}  using the formalism developed in \cite{Bunke:2018ty}.
Using the coarse cone construction $\cO^{\infty}$ \cite[Sec. 9.1]{equicoarse} and the (equivariant version of) \cite[Lemma 9.6]{ass} we define the  closed local homology functor
$K\cX^{G}\cO^{\infty}:G\UBC\to \Sp$ satisfying    (equivariant generalizations of) the axioms in \cite[3.12]{ass}.
By  \cite[Def. 4.14]{Bunke:2018ty}  the Dirac operator $\Dirac$ has a symbol class $\sigma(\Dirac)$ in 
$\Sigma^{-1} K\cX^{G}\cO^{\infty}(M)$.  The symbol class determines the index via \begin{equation}\label{werg09i0wergwr}
\partial \sigma^{G}(\Dirac)=\ind^{G}(\Dirac)\ ,
\end{equation} where $\partial$ is the cone boundary, see \cite[(9.1]{ass}. 
One can compare the Mayer-Vietoris boundary $\delta$ for $K\cX^{G}\cO^{\infty}$  applied to $M$ and the decomposition $(Z, M\setminus \inter(Z))$ with that for
$\R\otimes Z$ and $((-\infty,0]\otimes Z, [0,\infty)\otimes Z)$. Using the relative index theorem and the suspension 
formula \cite[Th, 10.4 \& Thm. 11.1]{indexclass}   we get\footnote{This is the sketchy step of the argument.} 
 the equality  $$\sigma^{G}( \partial \Dirac)=\delta \Res^{G}_{K}\sigma^{G}(\Dirac)\ .$$ 
 Since the cone boundary is a morphism of local homology theories  it is compatible with the Mayer-Vietoris boundaries. We conclude that 
\begin{eqnarray*}
r(\ind^{G}(\Dirac))&=&\delta( \Res^{G}_{K}(\ind^{G}(\Dirac)))
= \delta \Res^{G}_{K}\partial \sigma^{G}(\Dirac)
= \delta\partial  \Res^{G}_{K}\sigma^{G}(\Dirac)\\
&=&\partial  \delta  \Res^{G}_{K}\sigma^{G}(\Dirac)
=\partial  \sigma^{K}(\partial \Dirac)
=\ind^{K}(\partial \Dirac)\ ,
\end{eqnarray*}
as claimed.
\end{proof}

\begin{ex}
Assume that $D$ is of Laplace type and given by $D=\Dirac^{2}$ for some graded Dirac operator.
 Then we can decompose $D=D^{+}+D^{-}$, and we have the equality
 \begin{equation}\label{tgjoiwergergergwe}
[E_{D^{+}}(-1,\lambda)]-[E_{D^{-}}(-1,\lambda)]=\ind^{G}(\Dirac)\ .
\end{equation} 
In this case we can use \cref{qroigjhqoegrgegwergregwergre} in order to calculate at least the image of the difference 
$[E_{D^{+}}(-1,\lambda)]-[E_{D^{-}}(-1,\lambda)]$ under $r$. 
\end{ex}

\begin{ex}\label{qreoighjqorgergwrgwreg}
 We assume that 
 $\Dirac$ is the Dirac operator associated to a spin structure on $M$.
By the Schr\"odinger-Lichenerowicz formula we have $$\Dirac^{2}=\nabla^{*}\nabla+\frac{1}{4} s\ ,$$ where $s$ in $C^{\infty}(M)$ is the scalar curvature function.

\begin{kor}\label{weoigjwegergregegwr}
If $\ind^{K}(\partial \Dirac)\not=0$, then $ \inf_{Z }   s\le 0$.
\end{kor}
 \begin{proof}
 Assume that  
  $c:=\inf_{Z } s>0$. Then  
 the   spectra of the  Dirichlet extensions of
 $D^{\prime,\pm}$ are bounded below by $c$. 
 Then $c/2\not\in \sigma(D^{\prime,\pm})$ and by 
  \cref{ewgiowgregwergwrg}(\ref{feuivhiqvcq1})  this implies 
 $r([E_{D^{\pm}}(-1,c/2)])=0$.
  From \eqref{tgjoiwergergergwe} we conclude that 
 $r(\ind^{G}(\Dirac))=0$.
  By \cref{qroigjhqoegrgegwergregwergre} we get $\ind^{K}(\partial \Dirac)=0$.
 \end{proof}

 In view of $ \ind^{K}(\partial \Dirac)=\partial \sigma^{K}(\Dirac)$ we see that  $\ind^{K}(\partial \Dirac)$ only depends on the $K$-uniform bornological coarse structure associated to the $K$-invariant metric
 on $\partial Z$. If $\ind^{K}(\partial \Dirac)\not =0$, then by \cref{weoigjwegergregegwr}
 there does not exist a complete $G$-invariant metric 
 on $M$ whose restriction to $Z$ has uniformly positive scalar curvature 
 and induces the given $K$-uniform bornological coarse structure on $Z$.
 %
\end{ex}

\begin{rem}
In this {remark} we explain the relation of \cref{qreoighjqorgergwrgwreg} with the classical codimension-one obstruction against
positive scalar curvature in the version of \cite[Thm. 1.7]{Zeidler:2015aa}.

We let $\bar M$ be a connected closed $n$-dimensional spin manifold with a connected codimension-one submanifold $\bar N$ with trivialized normal bundle
such that $\pi_{1}(\bar N)\to \pi_{1}(\bar M)$ is injective. 
We then set $G:=\pi_{1}(\bar M)  $  and $K:=\pi_{1}(\bar N)$. We let 
$M$ be the universal covering of $\bar M$.
For  $\partial^{\infty} Z$ we choose a connected component of the preimage of $\bar N$ under the projection $M\to \bar M$. 
Since the normal bundle  of $\partial^{\infty}  Z$ is trivial, the submanifold $\partial^{\infty}  Z$  separates $M$ into  two components, and we let
$Z$ be the component such that the normal vectors are outward pointing. The subgroup $K$ preserves $Z$. 
 
 Since $G$ acts freely and cocompactly on $M$, and similarly, $K$ acts freely and cocompactly on $N$, we have equivalences
 \begin{align}\label{qwefouj0qewfqweffe}
K\cX^{G}(M)&\simeq K\cX^{G}(G_{can,min})\simeq K(C_{r}(G))\ , \\    K\cX^{K}(\partial^{\infty} Z  )&\simeq K\cX^{K}(K_{can,min})\simeq K(C_{r}(K))\ .\nonumber
\end{align} 
In the first line, the first equivalence is induced by the coarse equivalence $G_{can,min}\to M$ induced by the map
$G\to M$, $g\mapsto gm$, for any choice of a base point $m$ in $M$, and the second equivalence is  \cite[Prop. 8.2]{coarsek}. The second line is analoguous.
The $\alpha^{G}$-invariant   
 $\alpha^{G}(\bar M) $ of $\bar M$ in $  K_{n}(C_{r}(G))$ is defined as the image of    $\ind^{G}(\Dirac_{M})$ 
 under the identification \eqref{qwefouj0qewfqweffe}, {where $\Dirac_M$ is the spin Dirac operator.}
 Similarly,  $\alpha^{G}(\bar N)$ in $ K_{n-1}(C_{r}(G))$ is the image of    $\ind^{K}(\Dirac_{\partial Z})$  under  \eqref{qwefouj0qewfqweffe}.  
 Under the identifications \eqref{qwefouj0qewfqweffe} the map 
$$r_{*}:K\cX^{G}_{*}(M) \to K\cX_{*-1}^{K}(\partial Z) $$  corresponds to the map 
$$\Psi:K_{*}(C_{r}(G))\to K_{*-1}(C_{r}(K))$$ from the proof of \cite[Thm. 1.7]{Zeidler:2015aa}, and the equality
$$\Psi(\alpha^{G}(\bar M))=\alpha^{K}(\bar N)$$ shown in this reference 
is a {special} case {of} \cref{qroigjhqoegrgegwergregwergre}.

Note that  \cite{Zeidler:2015aa} has also a version for the maximal group $C^{*}$-algebras which at the moment
is not accessible by our general coarse methods.
\end{rem}

\color{black}

\section{Incoorporating local positivity}

Let $M$ and $D$ be as in  \cref{erguiqrgeqfweqef}. Let $P$ be a  $G$-invariant subset of $M$.

\begin{ddd}
We say that $D$ is uniformly locally positive on $P$ if there exists a $c$ in $(0,\infty)$ {such that the following holds.}
\begin{enumerate}
\item (Dirac case) We have $D^{2}=\nabla^{*}\nabla+R$ for some selfadjoint bundle endomorphism $R$ such that $R(m)\ge c^{2}$ for all $m$ in $P$.
\item (Laplace case) $D= \nabla^{*}\nabla+R$ for some selfadjoint bundle endomorphism $R$ such that
$R(m)\ge  c$ for all $m$ in $P$.
\end{enumerate}
The constant $c$ is called the positivity bound.
\end{ddd}

We let $M_{P}$ denote the $G$-bornological coarse space {$M$} with the metric coarse structure and the bornology 
$$\cB_{P}:=\{B\subseteq M \:|\: \mbox{$B\cap (M\setminus P)$ is bounded}\}\ .$$
 In this subsection we  let $Z_{P}$  be the {$K$-invariant} subset  $Z$ {of $M$} with the induced $K$-bornological coarse structure from $M_{P}$.

\begin{prop} For every $\varphi$ in $C_{0}((-c,c))$ 
 we have  $\varphi(D)\in C(M_{P},H,\phi)$ and  
 $\varphi(D^{\prime})\in C(Z_{P},H_{Z},\phi_{Z})$.
 \end{prop}
 \begin{proof}
 For the Dirac case
 the assertion $\varphi(D)\in C(\{M\setminus P\},H,\phi)$ is \cite[Lemma 2.3]{roe_psc_note}.  We now use the obvious inclusion $$C(\{M\setminus P\},H,\phi)\subseteq C(M_{P},H,\phi)\ .$$
 
Again the case of $D^{\prime}$ has the same proof.  

For the Laplace case, by positivity of $D$ and $D^{\prime}$ we can assume that $\varphi$ is even.
We then use the finite propagation speed of the wave operator family $\cos(t\sqrt{D})$ and argue similarly as in the Dirac case.
  \end{proof}

\begin{prop}\label{eirogjwoegwergerwf1}For every $\varphi$ in $C_{0}((-c,c))$ the operators $\varphi(D)$ and $\varphi(D^{\prime})$ are affiliated.
\end{prop}
\begin{proof}
The proof is the same as for  \cref{eirogjwoegwergerwf}.
  \end{proof}

{In the following we use the notation $r_{P}:K\cX^{G}(M_{P})\to \Sigma K\cX^{K}(\partial Z_{P})$ for the obstruction morphism  \eqref{eqwfpokqpowefqewfqwefewf} with the new bornology incorporated,
where $\partial Z_{P}$ is the big family $\partial Z$ whose members have  the bornological coarse structure induced from $M_{P}$. 
 The identity maps  of the underlying sets of $M$ and the members of the family  induce the vertical maps in the following    commutative  diagram  
$$\xymatrix{K\cX^{G}(M_{P})\ar[r]^{r_{P}}\ar[d]&\Sigma K\cX^{K}(\partial Z_{P})\ar[d]\\K\cX^{G}(M)\ar[r]^{r}&\Sigma K\cX^{K}(\partial Z)}\ .$$}

 Let $\sigma(D)$ denote the spectrum of $D$.
Let $\lambda$ be in $ (-\infty,c)\setminus \sigma(D)$.

\begin{prop}\label{eoirjgeggergegw} Assume that $D$ is of Laplace type.  \begin{enumerate} 
\item \label{feuivhiqvcqnnn}  We have
$E_{D}(-1,\lambda)\in C(M_{P},H,\phi)$. 
 \item  \label{feuivhiqvcq1nnn} If $r_{{P}} ([E_{D}(-1,\lambda)])\not=0$, then $\lambda\in \sigma(D^{\prime})$.
 \end{enumerate}
 
\end{prop}
\begin{proof}
Again this has the same proof as  \cref{ewgiowgregwergwrg}.
\end{proof}

If $D$ is  of Dirac type, then  we have an index class
$\ind^{G}(D)$ in $K\cX_{n}^{G}(M_{P})$. 

\begin{prop} \label{09qerug09egwgwerge}We assume that $D$ is of Dirac type. If $D^{\prime}$ exists, then we
 have $r_{{P}} (\ind^{G}(D))=0$.
\end{prop}
\begin{proof}  
The index class $\ind^{G}(D)$ is represented by the homomorphism $$C_{0}((-c,c))\ni\varphi\mapsto \varphi(D)\in C(M_{P},H,\phi)\ .$$ Similarly,  if $D^{\prime}$ exists, then $\ind^{K}(D^{\prime})$ in $K\cX^{K}_{n}(Z_{P})$ is represented by 
 the homomorphism $$C_{0}((-c,c))\ni\varphi\mapsto \varphi(D')\in C(Z_{P},H_{Z},\phi_{Z})\ .$$
  Since
$q(\varphi(D))=p(\varphi(D'))$ by  \cref{eirogjwoegwergerwf1} we conclude that
$q(\ind^{G}(D))=p(\ind^{K}(D^{\prime}))$. {We now argue as in the proof of \cref{fergoijergoewgergewrgreg}.}
 \end{proof}

\begin{kor}\label{wepogwegregegreg}
If $r_{{P}} (\ind^{G}(D))\not=0$, then $D^{\prime}$ does not admit a  local  selfadjoint elliptic boundary condition which is compatible with the Clifford action.
\end{kor}

{Since in general the map $K\cX_{*}^{K}(\partial Z_{P})\to K\cX_{*}^{K}(\partial Z)$  may have a kernel, the condition    $r_{P}(...)\not=0$  in \cref{eoirjgeggergegw}, \cref{09qerug09egwgwerge}, and \cref{wepogwegregegreg}     is   expected to be satisfied in more cases than the condition $r(...)\not=0$.  
}

%
 
  \section{Spectral projections of magnetic Laplacians}\label{ewoigjowergergerggw9}

 In this section we discuss the example of the  two-dimensional magnetic Laplacian. 

We consider  the smooth manifold $\R^{2}$  with coordinates $(x,y)$  and the standard Riemannian metric $dx^{2}+dy^{2}$. The group $\Z^{2}$ acts on $\R^{2}$ by translations $$((m,n),(x,y))\mapsto (x+m,y+n)\ .$$ In this way $\R^{2}$ becomes a $\Z^{2}$-bornological coarse space.  The embedding $\Z^{2}\to \R^{2}$ induces an equivalence of $\Z^{2}$-bornological coarse spaces $\Z^{2}_{can,min}\to \R^{2}$.  By \cite[Prop. 8.2]{coarsek} for any group $G$ we have an equivalence
$K\cX^{G}(G_{can,min})\simeq K(C^{*}_{r}(G))$, where $C^{*}_{r}(G)$ denotes the reduced group $C^{*}$-algebra of $G$. The case of $G=\Z^2$ gives the second equivalence in the following chain
 \begin{equation}\label{qwefefefewffqewfewfweqfqwf}
 K\cX^{\Z^{2}}(\R^{2})\stackrel{\simeq}{\leftarrow} K\cX^{\Z^{2}}(\Z^{2}_{can,min}) \simeq K(C_{r}^{*}(\Z^{2}))\stackrel{\text{Fourier}}{\simeq} K(C(\hat T^{2}))\stackrel{\substack{\text{Serre}\\\text{Swan}}}{\simeq} KU^{\hat T^{2}} \ ,
\end{equation}  
where  $\hat T^{2}$  denotes the dual group of $\Z^{2}$.  The decomposition $\Sigma^{\infty}_{+}\hat T^{2}\simeq S\oplus \Sigma S\oplus \Sigma S\oplus \Sigma^{2}S$ {(where $S$ is the sphere spectrum)} provides the decomposition
\begin{equation}\label{qwefefefewffqewfewfweqfqwf1}
KU^{ {\hat T^{2}}}\simeq KU\oplus \Sigma^{-1}KU\oplus \Sigma^{-1}KU\oplus \Sigma^{-2} KU\ .
\end{equation} 
 If we combine \eqref{qwefefefewffqewfewfweqfqwf} and \eqref{qwefefefewffqewfewfweqfqwf1} and apply $\pi_{0}$, then  we get an isomorphism  \begin{equation}\label{rhiorjtoherthertherthe}
(\dim,c_{1}): K\cX_{0}^{\Z^{2}}(\R^{2}) \cong  \Z\oplus \Z\ ,\end{equation}
where $\dim$ corresponds to the summand $\Z\cong \pi_{0}({KU} )$ contributed by the first summand  {in}  \eqref{qwefefefewffqewfewfweqfqwf1}, and $c_{1}$ corresponds to the summand $\Z\cong \pi_{0}(\Sigma^{-2}KU)$ contributed by the fourth summand in \eqref{qwefefefewffqewfewfweqfqwf1}. 

For a natural number $k\ge 1$ we  let $\Z^{2}$ act on the  total space of the trivial line bundle
$\R^{2}\times \C\to \R^{2}$ by  \begin{equation}\label{wergwegreglpwerg}((m,n),(x,y,z))\mapsto (x+m,y+n,e^{2\pi ik (my-nx)}z)\ .\end{equation} 
Let $W$ in $C^{\infty}(\R^{2})^{\Z^{2}}$ be a  real periodic function and set $w:=\|W\|_{\infty}$.
Then the partial differential operator 
\begin{equation}\label{wtiojeoirgewrg}D_{k}:= \Delta-4\pi i k (y\partial_{x}-x\partial_{y})+4\pi^{2}k^{2}(x^{2}+y^{2})-4\pi k + W \end{equation}
is invariant under this action, where $\Delta:=-(\partial_{x}^{2}+\partial_{y}^{2})$.

Assertion (\ref{werjweorvwerv}) in the following proposition states that ${D_{k}+w}$ is an example of an operator satisfying \cref{rwejgoiejgoiergwregwreg}.\ref{rojopwe}. The following has also been shown in \cite{dubnov}.

\begin{prop}\label{wtrhwergewgwerg}\mbox{}
\begin{enumerate}
\item \label{werjweorvwerv} The operator $D_{k}$ is of Laplace type, formally selfadjoint, and  {lower bounded by $-w$.}
\item \label{werjweorvwervergergergergerg} For every $\lambda\in {(w, {8}\pi k-w)}$ 
 the interval {$(-w {-1},\lambda)$} is a spectral interval for $D_{k}$. \item\label{iojiojoigrjwioergwergergwerg}  For every $\lambda\in {(w, {8}\pi k-w)}$ 
 we  have $[E_{D_{k}}(-w {-1},\lambda)]=(2k, {-}1)$ under the identification \eqref{rhiorjtoherthertherthe}.
\end{enumerate}
\end{prop}

{
\begin{rem}
In particular, if $W = 0$, then the spectrum of $D_k$ has $\{0\}$ as isolated point (it is an eigenvalue of infinite multiplicity), while the open interval $(0, 8 \pi k)$ is a spectral gap.
With a bootstrap argument (see \cite[\S 2.3]{LudewigThiangGapless}) one can in fact show that the entire spectrum is the set $\{ 8 \pi k n \mid n =0, 1, \dots\}$, with each spectral value being an eigenvalue of infinite multiplicity.
Of course, the introduction of $W$ will change this spectrum. 
However, by bounded perturbation theory, the perturbed spectrum will still be contained in the set of all points with distance at most $w$ from the unperturbed spectrum. 
\end{rem}
}

\begin{proof}
The idea is to identify $D_{k}$ with $\Dirac_{{L}}^{2,+} +W$ for a $\Z^{2}$-invariant Dirac-type operator $\Dirac_{{L}}$ on $\R^{2}$. The assertions are then deduced from the Weizenboeck formula for $\Dirac_{{L}}$ and the index theorem.
Note that $D_{k}$ is designed such that this interpretation exists.

We trivialize the spinor bundle of $\R^{2}$
so that $$S(\R^{2})\cong \R^{2}\times \C^{2}\ .$$ We let $c(\partial_{x})$ and $c(\partial_{y})$ be the Clifford multiplications on $S^{2}(\R^{2})$, and we let \begin{equation}\label{ergiojweorgwergwfefwefwe}z:=ic(\partial_{x})c(\partial_{y} )\end{equation} be the $\Z/2\Z$-grading operator. 
Then $S^{2}(\R^{2})^{\pm}$ {denote} the $\pm 1$-eigenbundles of $z$. They are both one-dimensional.
 
 We let $\Z^{2}$ act on
$S(\R^{2})$ by $(m,n)(x,y,v)\mapsto (x+m,y+n,v)$. By taking the quotient we obtain the spinor bundle $S(T^{2})\to T^{2} $ associated to the Spin structure   in which this bundle is trivialized. 

The Dirac operator on $S^{2}(\R^{2})$ is given by \begin{equation}\label{fdvsdfvfwrgvfvfs}
\Dirac=c(\partial_{x})\partial_{x}+c(\partial_{y})\partial_{y}\ .
\end{equation} 
We will obtain {a {generalized} Dirac operator} ${\Dirac_{L}}$ by twisting $\Dirac$ with an equivariant Hermitian line bundle $L\to \R^{2}$ with connection $\nabla^{L}$.
The underlying equivariant Hermitian vector bundle is the trivial bundle $\R^{2}\times \C\to \R^{2}$  with action of $\Z^{2}$  given by \eqref{wergwegreglpwerg}. The invariant connection is given by the formula
\begin{equation}\label{ewfqpokpoqkfqefqfewfewfq}
\nabla^{L}:=d-2\pi i k (xdy-ydx)\ .
\end{equation} 
 Its curvature form is   \begin{equation}\label{ergiojweorgwerg}R^{\nabla^{L}}=-4\pi i k dx\wedge dy\ .\end{equation}
  One furthermore calculates that the connection Laplacian is given by 
\begin{equation}\label{oijgowegrw}\nabla^{L,*}\nabla^{L}=\Delta-4\pi i k (y\partial_{x}-x\partial_{y})+4\pi^{2}k^{2}(x^{2}+y^{2})\ .\end{equation}
 We now form the 
 twisted Dirac operator \begin{equation}\label{wrgwrgwrevfr}\Dirac_{L}:=c(\partial_{x})\nabla^{L}_{\partial_{x}}+c(\partial_{y})\nabla^{L}_{\partial_{y}}\ .\end{equation} The  Weizenboeck formula states
$$\Dirac_{L}^{2}=\nabla^{L,*}\nabla^{L}+ c(\partial_{x})c(\partial_{y}) R^{\nabla^{L}}(\partial_{x},\partial_{y})\ .$$
  Using the explicit calculation  \eqref{ergiojweorgwerg} of the curvature $R^{\nabla}$  and the definition of the grading \eqref{ergiojweorgwergwfefwefwe} we can rewrite the Weizenboeck term as
$$c(\partial_{x})c(\partial_{y}) R^{\nabla^{L}}(\partial_{x},\partial_{y})= {-} 4\pi k z$$
 and get \begin{equation}\label{joiqjweofqwefq}\Dirac_{L}^{2}=\nabla^{L,*}\nabla^{L}-4\pi kz\ .\end{equation}
 The operator $\Dirac_{{L}}^{2}$ commutes with $z$, and we let $\Dirac_{L}^{2,\pm}$ be the restrictions to the $\pm 1$-eigenbundles of $z$. By combining \eqref{oijgowegrw} and \eqref{joiqjweofqwefq} we see that 
under the {canonical} identification of $S^2(\R^2)^\pm$ with {trivial complex line bundles}, we have
\begin{align*}
\Dirac_{L}^{2,+} &= D_{k} -W, \\
\Dirac_{L}^{2,-} &= D_{k} -W + 8 k \pi.
\end{align*}
It is clear that $D_{k}\cong \Dirac_{L}^{2,+}+W$ is {lower bounded by $-w$}. In view of  \eqref{joiqjweofqwefq} it is of Laplace type. This shows assertion (\ref{werjweorvwerv}). 

Since $\Dirac_{L}$ anticommutes with $z$ we have a bijection
$$\sigma(\Dirac_{L}^{2,+})\setminus \{0\}=\sigma(\Dirac_{L}^{2,-})\setminus \{0\}\ .$$
Since  $\Dirac_{L}^{2,+}\ge 0$ we have 
$\Dirac_{L}^{2,-}\ge 8\pi k$. Consequently, the interval $(0,8\pi k)$ belongs to the resolvent set of $\Dirac_{L}^{2,+}$
By bounded perturbation theory, this implies that $(w,8\pi k-w)$ belongs to the resolvent set of 
$D_{k}=\Dirac_{L}^{2,+}+W$.
  Hence we conclude assertion (\ref{werjweorvwervergergergergerg}).
  
We further conclude that $0$ is an isolated eigenvalue of $\Dirac_{L}^{2,+}$. 
Using  pertubation theory we conclude that 
$$[E_{D_{k}}(-w {-1},\lambda)]=[E_{\Dirac_{L}^{2,+}}(\{0\})]$$ 
in 
{$K\cX_0^{\Z^{2}}(\R^{2})$} 
 for all $\lambda$ in $(w,8\pi k-w)$.

 In order show assertion  (\ref{iojiojoigrjwioergwergergwerg}) we must 
 calculate the $K$-theory class of the representation of $C^{*}_{r}(\Z^{2})$ on $\ker(\Dirac^{2,+}_{L})$.
Under  the Fourier transformation isomorphism $C^{*}_{r}(\Z^{2})\cong C(T^{2})$ this representation corresponds to the representation of $C(T^{2})$ on the space of sections of the kernel bundle of a family of Dirac operators parametrized by $T^{2}$ which we will describe the following.

The $\Z^{2}$-Hilbert space $L^{2}(\R^{2},S(\R^{2})\otimes L)$ can be identified with a direct integral 
$$\int_{\hat T^{2}} V_{\chi} d\chi$$ for the continuous field of Hilbert spaces $((V_{\chi})_{\chi\in \hat T},\cV)$.
Here $V_{\chi}$ is the subspace of $L^{2}_{\loc}(\R^{2},S(\R^{2})\otimes L)$ of functions on which $\Z^{2}$ acts with character
$\chi$, i.e., of sections $f$ of $S(\R^{2})$ such that
$$e^{2\pi i k(my-nx)}f(x-m,y-n)=\chi(m,n)f(x,y)$$ for all $(m,n)$ in $ \Z^{2}$,
with the scalar product  given by the $L^{2}$-scalar product of the the restriction to a fundamental domain.
The continuous structure  is given by the subspace $\cV$ of sections $g: \hat T^{2}\ni\chi\mapsto g(\chi)\in V_{\chi}$ such that the section$\int_{\hat T^{2}} g(\chi)\chi$ of $S(\R^{2}\otimes L)$ is smooth with compact support.

The operator $\Dirac_{L}$ is the given by a family of operators $(\Dirac_{L,\chi})_{\chi\in \hat  T^{2}}$, where
$\Dirac_{L,\chi}$ acts on $V_{\chi}$ by the formula \eqref{wrgwrgwrevfr}.  Since $\Dirac_{L}$ preserves smooth compactly supported sections the family 
 $(\Dirac_{L,\chi})_{\chi\in  \hat T^{2}}$ preserves the subspace $\cV$.
We must calculate the kernel bundle of this family.

We will now describe this family of operators as a family of twisted Dirac operators.
To this end we consider the Poincar\'e bundle $\bar P\to   T^{2}\times \hat T^{2}$. We let $(s,t)$ be coordinates on $\hat  T^{2}$ such that $$\chi(s,t):\Z^{2}\to U(1)\ , \quad \chi(s,t)(m,n)=e^{2\pi i (sm+tn)}\ .$$
  The bundle $  \bar P\to  T^{2}\times  \hat   T^{2}$ is obtained as quotient of the trivial bundle
$P:=\R^{2}\times \hat T^{2}\times \C\to \R^{2}\times \hat T^{2}$ by $\Z^{2}$, where $\Z^{2}$ acts on the domain by $$\psi(m,n)(x,y,s,t,z):=(x+m,y+n,s,t, e^{2\pi i (sm+tn)}z)\ .$$ 
On $P$ we consider the invariant  connection $$\nabla^{P}:=d+2\pi i (xds+ydt)\ .$$ It induces a connection $\nabla^{\bar P}$ on the quotient $\bar P$.
Its curvature is given by 
\begin{equation}\label{erjioejbioejrvowverv} R^{\nabla^{P}}=2\pi i (dx\wedge ds+dy\wedge dt)\ .\end{equation}

We let $\bar L\to T^{2}$ be the quotient of the bundle $ L\to \R^{2}$, and $\Dirac_{\bar L}$ be the resulting twisted  Dirac operator on $T^{2}$. We can consider this operator as a constant family on the bundle $T^{2}\times \hat T^{2}\to \hat T^{2}$. 
We now form the non-constant family $\Dirac_{\bar L\otimes \bar P}$ by twisting $\Dirac_{\bar L}$ further with $\bar P$.
{On the fibre over $\chi$ in $\hat T^{2}$ the operator  $\Dirac_{\bar L\otimes \bar P}$ identifies with $\Dirac_{L,\chi}$. 
 Hence the index bundle of $\Dirac_{\bar L\otimes \bar P}$ equals the kernel bundle of the family $(\Dirac_{L,\chi})_{\chi\in  \hat T^{2}}$.}


We can now apply the Atiyah-Singer index theorem for families in order to calculate the Chern character of the index bundle $\ch(\ind(\Dirac_{\bar L\otimes \bar P}))$ in $H^{*}(\hat T^{2})$ of this family. We get 
$$\ch(\ind(\Dirac_{\bar L\otimes \bar P}))=\int_{T^{2}\times \hat T^{2}/\hat T^{2}} \ch(\bar L\otimes \bar P)\ .$$
We calculate in de Rham cohomology. Using the formulas for the curvatures \eqref{ergiojweorgwerg} and  \eqref{erjioejbioejrvowverv} and the formulas $c_{1}(\nabla)=\frac{i}{2\pi} R^{\nabla}$ and $\ch(\nabla)=e^{c_{1}(\nabla)}$ for a line bundle we have
\begin{eqnarray*} 
\ch(\nabla^{L}\otimes \nabla^{P})&=&1+2k(dx\wedge dy)-(dx\wedge ds+dy\wedge dt)\\&& - (dx\wedge dy\wedge ds\wedge dt)\ .
\end{eqnarray*}
We conclude that
$$\int_{T^{2}\times \hat T^{2}/\hat T^{2}} \ch(\bar L\otimes \bar P)=2k-[ds\wedge dt]\ .$$
This implies assertion (\ref{iojiojoigrjwioergwergergwerg}).
\end{proof}

  In the remainder of this subsection we show that the magnetic Laplacian $D_{k}$ provides a non-trivial example for the theory developed in \cref{erguiqrgeqfweqef}.
We consider a codimension-zero submanifold $Z$ of $\R^{2}$
with smooth boundary $\partial Z$ such that $Z$ and $\R^{2}\setminus Z$ are flasque, e.g.,  a half space.
We then consider the morphism  $r:K\cX^{\Z^{2}}(\R^{2})\to K\cX(\partial Z)$ from \eqref{wegoijwegerwregwg} for the trivial group $K$.
 \begin{lem}\label{weoirgjowggergregwergreg9}
For $\lambda$ in $(w,8\pi k-w)$
we have $r([E_{D_{k}}(-w {-1},\lambda)])\not=0$.
\end{lem}
\begin{proof}
By our flasqueness assumption on $Z$ and $ \R^{2}\setminus Z$ the Mayer-Vietoris boundary $\delta$ in \eqref{FibreSequenceZXZ}
is an equivalence. It follows that $r$ is equivalent to the transfer map 
\begin{equation}\label{werfwrefrwewferferfwrfref}
c^{\Z^{2}}:K\cX^{\Z^{2}}(\R^{2})\to K\cX(\Res^{\Z^{2}}(\R^{2}))\ .
\end{equation} 
It thus suffices to show that the image of the class 
$[E_{D_{k}}(-w {-1},\lambda)]$ in $ K\cX_{0}^{\Z^{2}}(\R^{2})$ under  \eqref{werfwrefrwewferferfwrfref}  is non-trivial.  

In principle this immediately follows  from \cref{wtrhwergewgwerg}.\ref{iojiojoigrjwioergwergergwerg} and the isomorphism 
\eqref{rhiorjtoherthertherthe} if one could identify the transfer morphism with the component $c_{1}$ up to sign.
This fact  is actually not completely obvious but could be deduced a posteriori from our calculations.
We prefer to give an argument which is independent of the calculation  in \cref{wtrhwergewgwerg}.\ref{iojiojoigrjwioergwergergwerg}
and is also applicable in other situations.

 The first equality in  the chain
$$[E_{D_{k}}(-w {-1},\lambda)]=[E_{\Dirac_{L}^{2,+}}(\{0\})]\stackrel{!}{=}\ind^{\Z^{2}}(\Dirac_{L})$$
has been shown  in the proof   of \cref{wtrhwergewgwerg}.  
The symbol $\ind^{\Z^{2}}(\Dirac_{L}) $ denotes  the equivariant coarse index  in {$K\cX^{\Z^{2}}_0(\R^{2})$} of the Dirac operator
$\Dirac_{L}$ on $\R^{2}$ as in \cite[Def. 9.5]{indexclass} (with full support).
  For the equality marked by $!$ note that,  
{as seen in the proof of \eqref{wtrhwergewgwerg}, the operator}
 $\Dirac_{L}^{2,-}$ is invertible and $0$  is an isolated point of the spectrum of  $\Dirac_{L}$.
We also use Bott periodicity in order to identify the index of the Dirac operator which is a $K$-theory class in degree $2$ with the degree-$0$ class represented by the spectral projection.

The transfer map
\eqref{werfwrefrwewferferfwrfref} sends  $\ind^{\Z^{2}}  (\Dirac_{L}) $ to the index
$\ind  ( \Dirac_{L}) $ in $K\cX(\Res^{\Z^{2}}(\R^{2}))$ of the same operator $\Dirac_{L}$ without equivariance, see \cite[Rem. 10.6]{coarsek}.  
In order to show that $ \ind  ( \Dirac_{L} )\not=0$ we 
  argue that \begin{equation}\label{ewfqwfpokpqwfewfqwf}
\ind  (\Dirac_{L})= \ind  (\Dirac)
\end{equation}
with $\Dirac$ as in \eqref{fdvsdfvfwrgvfvfs}
  and then  use the  fact $ \ind   (\Dirac )$ is a generator of $ K\cX_{2}(\R^{2})\cong \Z$.

The idea for showing the equality  \eqref{ewfqwfpokpqwfewfqwf} is   that $\ind (\Dirac_{L}) $  is  the image of the symbol of $ \Dirac_{L} $ under the coarse assembly map, and that the symbol classes of $ \Dirac_{L} $ and $  \Dirac $ coincide as these operators are lower order perturbations of each other. 
Since $\R^{2}$ is non-compact some care is needed in the details.  
The problem is that $\Dirac_{L}-\Dirac$ is zero order, but unbounded.  Therefore it is not clear that the linear interpolation between them preserves the index.
We will give a completely formal argument using the coarse interpretation of index theory developed in  \cite[Def. 4.14]{Bunke:2018ty}.
Alternatively one could argue as in \cite[Lem. 3.2]{LudewigThiangGapless}.

   
       In the following we omit the symbol $\Res^{\Z^{2}}$.  We will use  the local homology theory $K\cX\cO^{\infty}:\UBC\to \Sp$ from \cite[Lem. 9.6]{ass},
  where $\UBC$ is the category of 
   uniform bornological coarse spaces. This theory is well-defined since $K\cX$ is strong by \cite[Prop. 6.4]{coarsek}.
   The cone boundary
   $$\partial :\Sigma^{-1}K\cX\cO^{\infty} \to  K\cX  $$ (see \cite[(9.1)]{ass})
   is a natural transformation of local homology theories, where we consider  $\Sigma^{-1}K\cX$ as a local homology 
   via the forget functor $\UBC\to \BC$.
     According to \cite[Def. 4.14]{Bunke:2018ty} 
 any generalized Dirac operator    $\Dirac'$ on $\R^{2}$ gives rise to a symbol class $\sigma(\Dirac')$ in $\pi_{1}(K\cX\cO^{\infty}(\R^{2}))$ such that
 $\ind  (\Dirac')=\partial(\sigma(\Dirac'))$.
 By   \cite[Prop. 11.23]{ass} the functor $K\cX\cO^{\infty}$ behaves  on finite-dimensional manifolds like a locally finite homology theory.  In particular, if  $B$ is the unit ball in $\R^{2}$, then the canonical map 
 \begin{equation}\label{adsvasvvsvdvdavdsvdsavavasdvasvasvsdv}
e:K\cX\cO^{\infty}(\R^{2}) \to K\cX\cO^{\infty}(\R^{2},\R^{2}\setminus B)\simeq \Cofib (K\cX\cO^{\infty}(\R^{2}\setminus B) \to  K\cX\cO^{\infty}(\R^{2}))
\end{equation}  is an equivalence.

 Note that $\sigma(\Dirac_{L})$ is given by the index of a Dirac  operator  on a geometric version of the cone
which is locally derived from $\Dirac_{L}$.   
 By the coarse relative index theorem  \cite[Thm. 10.4]{indexclass} 
 the class $e(\sigma(\Dirac_{L}))$  only depends on the restriction
 of $\Dirac_{L}$ to $B$. Let $\psi$ be in $C_{c}(\R^{2})$ such that $\psi_{|B}=1$ and $\psi_{|\R^{2}\setminus 2B}=0$.
Then we consider the family of connections (compare with \eqref{ewfqpokpoqkfqefqfewfewfq})
$$\nabla^{L}_{t}:=d- (1-t\psi (x,y)) 2\pi i k (xdy-ydx) $$ and the associated Dirac operator $\Dirac_{L_{t}}$.
We have $\Dirac_{L_{0}}=\Dirac_{L}$ and $(\Dirac_{L_{1}})_{|B}=\Dirac_{|B}$. 
 Since the family $D_{L_{t}}$ is constant outside of  the compact subset $2B$ this perturbation does not alter the 
coarse index, i.e. we have
$\ind  (\Dirac_{L_{t}})=\ind   (\Dirac_{L})$ for all $t$ in $[0,1]$.
On the other hand, $e(\sigma(\Dirac_{L_{1}}))=e(\sigma(\Dirac))$ and hence
$ \sigma(\Dirac_{L_{1}})=\sigma(\Dirac)$ by the injectivity of $e$ in  \eqref{adsvasvvsvdvdavdsvdsavavasdvasvasvsdv}.
This implies the equality \eqref{ewfqwfpokpqwfewfqwf} in view of the chain   
\begin{equation*}
 \ind  (\Dirac_{L})
 =    \ind (\Dirac_{L_{1}}) 
 =\partial  \sigma(\Dirac_{L_{1}})
 =\partial  \sigma(\Dirac)
 = \ind  (\Dirac )\ . \qedhere
 \end{equation*}
   \end{proof}

 Combining \cref{wtrhwergewgwerg}, \cref{weoirgjowggergregwergreg9} and  \cref{ewgiowgregwergwrg} we can now conclude:
 \begin{kor}\label{qeroighjwoergwregrweg9}
 If $D_{k}^{\prime}$ is   the Dirichlet extension (or any other selfadjoint extension determined by a local boundary condition) of the restriction of $D_{k}$ to $Z$, then  $[w,8\pi k-w]\subseteq \sigma(D_{k}^{\prime}) $.
 \end{kor}

Examples for $Z$ are
half spaces or spaces of the form 
$\{(x,y)\:|\: y\le f(x)\}$ for  some smooth function $f:{\R}\to \R$.
Furthermore, we can allow any bounded  perturbation of a space of this form, e.g.
  the   half space $\{(x,y)\:|\: y\le 0\}$ together with the union of $1/3$-balls at all points
$(n,k)$ for  all 
$n$ in $\nat$ and fixed $k$ in $\nat$.

 \begin{ex}
 We provide another example of a $K$-theory class which can be detected using the restriction morphism. The unitary $u$ in $C^{*}_{r}(\Z^{2})$ given by the element $(1,0)$ of $\Z^{2}$ 
  represents a class
 $[u]\in K_{1}(C_{r}^{*}(\Z^{2}))\simeq K\cX_{1}^{\Z^{2}}(\Z^{2}_{can,min})$.
We consider the group $  \Z$ as a subgroup of $\Z^{2}$ embedded  by $n\mapsto (0,n)$ and the $\Z$-invariant subspace $Z=\{(n,m)\mid n\ge 0\}$ of $\Z^{2}_{can,min}$.
The inclusion $\Z_{can,min}\to \Res^{\Z^{2}}_{\Z}(\Z^{2}_{can,min})$
induces an equivalence
$K\cX^{\Z}(\Z_{can,min})\simeq K\cX^{\Z}(\partial Z)$. 
The morphism  $r$ from \cref{wegoijwegerwregwg} can therefore be interpreted as a morphsim 
$$r: K(C_{r}^{*}(\Z^{2}))\simeq K\cX^{\Z^{2}}(\Z^{2}_{can,min})\to  \Sigma K\cX^{\Z}(\partial Z) \simeq  \Sigma K\cX^{\Z}(\Z_{can,min})  \ .$$ 
One can now check by an explicit calculation that
$r([u])$ is a generator of the cyclic group  $ K\cX_{0}^{\Z}(\Z_{can,min})$.
 
 This idea can be expanded in order to obtain a 
coarse-geometric proof of the fact  shown in  \cite{elliott-natsume} that for a general
group $G$ the canonical map $G_{\ab}\to K_{1}(C_{r}^{*}(G))$ is injective.
 \end{ex}

\bibliographystyle{alpha}
\bibliography{unik}

\end{document}